\documentclass[a4paper,12pt]{article}
\usepackage[english]{babel}
\usepackage{tikz}
\usetikzlibrary{matrix,arrows,decorations.markings}
\usepackage{amsmath,amsfonts,amssymb,amsthm,url,textcomp}
\usepackage{csquotes}
\usepackage{a4wide}
\usepackage[numbers]{natbib}
\usepackage{fancyhdr}
\usepackage{anyfontsize}
\usepackage{hyperref}
\usepackage{hhline}
\usepackage{leftidx}

\allowdisplaybreaks

\newtheorem{theorem}{Theorem}\numberwithin{theorem}{section}

\newtheorem{proposition}[theorem]{Proposition}
\newtheorem{notation}[theorem]{Notation}

\newtheorem{question}[theorem]{Question}
\newtheorem{problem}[theorem]{Problem}
\newtheorem{theoremm}{Theorem}\numberwithin{theoremm}{subsection}
\newtheorem{deffinition}[theoremm]{Definition}
\newtheorem{lemmma}[theoremm]{Lemma}

\newtheorem{nottation}[theoremm]{Notation}
\newtheorem{propposition}[theoremm]{Proposition}

\numberwithin{theoremmm}{subsubsection}

\theoremstyle{remark}

\newtheorem{remmark}[theoremm]{Remark}

\newcommand{\Rad}{\operatorname{Rad}}
\newcommand{\Aut}{\operatorname{Aut}}
\newcommand{\Alt}{\mathcal{A}}
\newcommand{\PSL}{\operatorname{PSL}}

\newcommand{\lcm}{\operatorname{lcm}}

\newcommand{\ord}{\operatorname{ord}}
\newcommand{\Sym}{\mathcal{S}}

\newcommand{\C}{\operatorname{C}}

\newcommand{\Soc}{\operatorname{Soc}}

\newcommand{\id}{\operatorname{id}}
\newcommand{\N}{\operatorname{N}}

\newcommand{\Out}{\operatorname{Out}}
\newcommand{\e}{\mathrm{e}}

\newcommand{\PGL}{\operatorname{PGL}}

\newcommand{\PSU}{\operatorname{PSU}}
\newcommand{\PSp}{\operatorname{PSp}}

\newcommand{\D}{\operatorname{D}}

\renewcommand{\O}{\mathcal{O}}

\newcommand{\IN}{\mathbb{N}}

\newcommand{\Sp}{\operatorname{Sp}}

\newcommand{\IF}{\mathbb{F}}
\newcommand{\Stab}{\operatorname{Stab}}

\newcommand{\IZ}{\mathbb{Z}}
\renewcommand{\o}{\operatorname{o}}

\newcommand{\MAOL}{\operatorname{MAOL}}
\newcommand{\maol}{\operatorname{maol}}
\newcommand{\h}{\mathfrak{h}}
\newcommand{\bcpc}{\operatorname{bcpc}}
\newcommand{\type}{\operatorname{type}}
\newcommand{\IR}{\mathbb{R}}
\newcommand{\MCS}{\operatorname{MCS}}
\newcommand{\Inndiag}{\operatorname{Inndiag}}
\newcommand{\IP}{\mathbb{P}}
\newcommand{\PGU}{\operatorname{PGU}}
\newcommand{\POmega}{\operatorname{P}\Omega}
\newcommand{\GO}{\operatorname{GO}}
\newcommand{\PGO}{\operatorname{PGO}}
\newcommand{\Outdiag}{\operatorname{Outdiag}}
\newcommand{\CT}{\operatorname{CT}}
\newcommand{\supp}{\operatorname{supp}}
\newcommand{\pcal}{\mathfrak{p}}

\begin{document}

\title{Finite groups with a large automorphism orbit}

\author{Alexander Bors\thanks{School of Mathematics and Statistics, University of Western Australia, Crawley 6009, WA, Australia. \newline E-mail: \href{mailto:alexander.bors@uwa.edu.au}{alexander.bors@uwa.edu.au} \newline The author is supported by the Austrian Science Fund (FWF), project J4072-N32 \enquote{Affine maps on finite groups}. \newline 2010 \emph{Mathematics Subject Classification}: Primary: 20D45. Secondary: 20D05, 20D60, 60C05. \newline \emph{Key words and phrases:} Finite groups, Automorphism orbits, Composition factors.}}

\date{\today}

\maketitle

\abstract{We study the nonabelian composition factors of a finite group $G$ assumed to admit an $\Aut(G)$-orbit of length at least $\rho|G|$, for a given $\rho\in\left(0,1\right]$. Our main results are the following: The orders of the nonabelian composition factors of $G$ are then bounded in terms of $\rho$, and if $\rho>\frac{18}{19}$, then $G$ is solvable. On the other hand, for each nonabelian finite simple group $S$, there is a constant $c(S)\in\left(0,1\right]$ such that $S$ occurs with arbitrarily large multiplicity as a composition factor in some finite group $G$ having an $\Aut(G)$-orbit of length at least $c(S)|G|$.}

\section{Introduction}\label{sec1}

\subsection{Motivation and main results}\label{subsec1P1}

The notion of an automorphism as a formalization of a \enquote{symmetry} of an object is ubiquitous in mathematics, and quite a bit of research across mathematical disciplines is concerned with studying \enquote{highly symmetric} objects $X$, which usually means at least that the natural action of the automorphism group $\Aut(X)$ on $X$ is transitive, though often, even more than that is assumed. As examples, we mention vertex-transitive graphs \cite[Definition 4.2.2, p.~85]{BW79a} and Cayley graphs \cite[Definition 5.3.1 and Theorem 5.3.2, pp.~115f.]{BW79a} from graph theory, block-transitive \cite{CP93a,CP93b} and flag-transitive \cite{Hub09a} designs from combinatorics, and flag-transitive finite projective planes \cite{Tha03a} from geometry.

In trying to adapt these studies of \enquote{highly symmetric} objects $X$ to the case where $X$ is a finite group $G$, one encounters the fundamental problem that unless $G$ is trivial, the action of $\Aut(G)$ on $G$ is intransitive. It is therefore necessary to weaken the symmetry condition on $G$, and a natural way to do so is by only assuming that a certain positive proportion of the elements of $G$ (instead of \emph{all} of $G$) is contained in an $\Aut(G)$-orbit. This motivates the study of the following concepts:

\begin{deffinition}\label{maolDef}
Let $G$ be a finite group.

\begin{enumerate}
\item We set $\MAOL(G):=\max_{g\in G}{|g^{\Aut(G)}|}$, the \emph{(absolute) maximum automorphism orbit length of $G$}.
\item We set $\maol(G):=\frac{1}{|G|}\MAOL(G)\in\left(0,1\right]$, the \emph{relative maximum automorphism orbit length of $G$}, or the \emph{maximum automorphism orbit proportion of $G$}.
\end{enumerate}
\end{deffinition}

Our main results, listed in Theorem \ref{mainTheo} below, are concerned with the nonabelian composition factors of finite groups $G$ which satisfy $\maol(G)\geq\rho$ for some given $\rho\in\left(0,1\right]$.

\begin{theoremm}\label{mainTheo}
The following hold:

\begin{enumerate}
\item A finite group $G$ with $\maol(G)>\frac{18}{19}$ is solvable.
\item There is a function $f:\left(0,1\right]\rightarrow\left(0,\infty\right)$ such that for all $\rho\in\left(0,1\right]$ and each finite group $G$ with $\maol(G)\geq\rho$, the orders of the nonabelian composition factors of $G$ are at most $f(\rho)$.
\item For each nonabelian finite simple group $S$, there is a constant $c(S)\in\left(0,1\right]$ such that for each positive integer $N$, there is a finite group $G$ with $\maol(G)\geq\rho$ and such that $S$ occurs as a composition factor of $G$ with multiplicity at least $N$.
\end{enumerate}
\end{theoremm}

\begin{remmark}\label{mainRem}
Some remarks concerning Theorem \ref{mainTheo}:

\begin{enumerate}
\item The constant $\frac{18}{19}$ in Theorem \ref{mainTheo}(1) is probably not optimal; see the concluding remarks (Section \ref{sec4}, Proposition \ref{smallExProp} and Question \ref{ques1}) for the conjectured optimal value.
\item By Theorem \ref{mainTheo}(3), in general, the multiplicities of the nonabelian composition factors $S$ of a finite group $G$ satisfying $\maol(G)\geq\rho$ cannot be bounded in terms of $\rho$, even if the bound is allowed to depend on $S$ as well. This is in sharp contrast to \cite[Theorem 1.1.1(3)]{Bor17a}, which implies that if there is a cyclic subgroup of $\Aut(G)$ admitting an orbit on $G$ of length at least $\rho|G|$, then the index $[G:\Rad(G)]$, where $\Rad(G)$ is the \emph{solvable radical} of $G$ (the largest solvable normal subgroup of $G$), is bounded in terms of $\rho$ (which is equivalent to having a bound in terms of $\rho$ on both the orders and the multiplicities of the nonabelian composition factors of $G$).
\item We also note that in contrast to \cite[Theorem 1.1.1(1)]{Bor17a}, there is no $\rho\in\left(0,1\right)$ such that a finite group $G$ with $\maol(G)>\rho$ is necessarily abelian, since (as can be easily checked) for each prime $p>2$, the nonabelian group of order $p^3$ and exponent $p$ admits an automorphism orbit of proportion $1-\frac{1}{p}$.
\end{enumerate}
\end{remmark}

\subsection{Overview of the paper}\label{subsec1P2}

Each subsection of Section \ref{sec2} is dedicated to the proof of one of four lemmas that will be used for proving Theorem \ref{mainTheo}. Lemma \ref{mainLem1} from Subsection \ref{subsec2P1} serves to essentially reduce the problem to bounding automorphism orbit proportions in finite groups $H$ such that $S^n\leq H\leq\Aut(S^n)$ for some nonabelian finite simple group $S$ and some positive integer $n$. Such bounds are then obtained in Lemma \ref{mainLem2} from Subsection \ref{subsec2P2}. They are products of values of the probability mass functions of certain multinomial distributions, and Lemma \ref{mainLem3} from Subsection \ref{subsec2P3} provides an upper bound on these factors, which in turn reduces the proof of Theorem \ref{mainTheo}(1) and of most of Theorem \ref{mainTheo}(2) to bounding, for each nonabelian finite simple group $S$, the maximum proportion $\h(S)$ (see also Notation \ref{hNot}) of an $\Aut(S)$-conjugacy class in one of the cosets of $S$ in $\Aut(S)$. Using the classification of the finite simple groups, such bounds are obtained in Lemma \ref{mainLem4} from Subsection \ref{subsec2P4}.

Section \ref{sec3} then consists of the actual proof of Theorem \ref{mainTheo}, and in Section \ref{sec4}, we conclude the paper with some interesting open questions and problems for further research.

\subsection{Notation and terminology}\label{subsec1P3}

We denote by $\IN$ the set of natural numbers (including $0$) and by $\IN^+$ the set of positive integers. The element-wise image of a set $M$ under a function $f$ is denoted by $f[M]$, and the identity funtion on $M$ by $\id_M$. The symbol $\phi$ denotes Euler's totient function, $\e$ denotes Euler's constant, and for a ring $R$, $R^{\ast}$ denotes the group of units of $R$. For a prime power $q$, the finite field with $q$ elements is denoted by $\IF_q$, and its algebraic closure by $\overline{\IF_q}$. The greatest common divisor of $a,b\in\IN^+$ is denoted by $\gcd(a,b)$ or simply by $(a,b)$ if there is no risk of confusing it with the ordered pair of $a$ and $b$ (such as in the formulas for outer automorphism group orders of simple Lie type groups in the proof of Lemma \ref{mainLem4}).

The symmetric and alternating groups of degree $m$ will be denoted by $\Sym_m$ and $\Alt_m$ respectively. The term \emph{characteristic quotient (of a group)} means \enquote{quotient by a characteristic subgroup}. If $G$ is a group and $g$ is an element or subgroup of some group $H\geq G$, then we denote by $\C_G(g)$ resp.~$\N_G(g)$ the centralizer resp.~normalizer of $g$ in $G$. When $\sigma$ is a permutation on a finite set $M$ and $O\subseteq M$ is an orbit of the natural action of $\langle\sigma\rangle$ on $M$, we view $O$ as the support of the associated cycle $\zeta$ from the representation of $\sigma$ as a disjoint product of cycles, using the notation $\supp(\zeta):=O$. Note that this differs slightly from the usual notion of the support of a permutation, under which $1$-cycles would have empty support. For a finite group $G$, we denote by $\Soc(G)$ the \emph{socle of $G$} (the subgroup of $G$ generated by the minimal nontrivial normal subgroups of $G$), by $\zeta G$ the center of $G$, and by $\MCS(G)$ the minimum size of an element centralizer in $G$. An automorphism $\alpha$ of a nonabelian finite simple group $S$ is called \emph{out-central} if and only if its image under the canonical projection $\Aut(S)\rightarrow\Out(S)$ lies in $\zeta\Out(S)$. For an element $\vec{g}=(g_1,\ldots,g_n)\sigma$ of the wreath product $G\wr Sym_n=G^n\rtimes\Sym_n$, we call $(g_1,\ldots,g_n)\in G^n$ the \emph{tuple part} and $\sigma\in\Sym_n$ the \emph{permutation part of $\vec{g}$ (with respect to the fixed wreath product decomposition of $G\wr\Sym_n$)}.

As for our notation of the finite simple groups of Lie type, we follow the approach taken in \cite[Section 3, pp.~104f.]{Har92a}, so that $\leftidx{^t}{X_r(p^{f\cdot t})}$, where the pre-superscripted $t$ is usually omitted if it is $1$, denotes $O^{p'}(X_r(\overline{\IF_p})_{\sigma})$, the subgroup of the simple Chevalley group (i.e., simple linear algebraic group of adjoint type) $X_r(\overline{\IF_p})$ generated by the $p$-elements of the fixed point subgroup of the Lang-Steinberg map (\enquote{Frobenius map} in the terminology of \cite[p.~104]{Har92a}) $\sigma$ on $X_r(\overline{\IF_p})$, and the relationship between $\sigma$ and the parameters $t=t(\sigma)$ and $f=f(\sigma)$ is as follows: Let $B$ be any $\sigma$-invariant Borel subgroup of $X_r(\overline{\IF_p})$, and let $T$ be any $\sigma$-invariant maximal torus of $X_r(\overline{\IF_p})$ contained in $B$. Then $t$ is the unique smallest positive integer (independent of the choice of $B$ and $T$) such that the $t$-th power of the map $\sigma^{\ast}$ on the character group $X(T)$ induced by $\sigma$ is a positive integral multiple of $\id_{X(T)}$, and $f\in\IN^+/2=\{\frac{1}{2},1,\frac{3}{2},\ldots\}$ is such that $\sigma^{\ast}=p^f\sigma_0$ with $\sigma_0^t=\id_{X(T)}$; $f$ also does not depend on the choice of $B$ and $T$. So $p^f=q(\sigma)$ in the notation of \cite{Har92a}, which is also a notation we will be using, and $f\in\IN^+$ unless $\leftidx{^t}{X_r(p^{f\cdot t})}$ is one of the Suzuki or Ree groups, in which case $f$ is half of an odd positive integer. For us, a \emph{finite simple group of Lie type} is by definition any group of the form $\leftidx{^t}X_r(p^{ft})$, even if it is not a simple group (such as $A_1(2)$). We say that $\leftidx{^t}X_r(p^{ft})$ is of \emph{untwisted Lie rank} $r$; with a few small exceptions (such as $A_1(7)\cong A_2(2)$), each finite simple group of Lie type has precisely one untwisted Lie rank. In the context of finite simple groups of Lie type, the terms \enquote{graph automorphism}, \enquote{field automorphism} and \enquote{graph-field automorphism} (the last meaning \enquote{product of a field and a graph automorphism}) and the associated notations $\Phi_S$ and $\Gamma_S$ are used as explained in \cite[p.~105]{Har92a}. Moreover, as in \cite{GLS98a}, $\Inndiag(S)$ denotes the inner diagonal automorphism group of $S$ (so $\Inndiag(\leftidx{^t}X_r(p^{f\cdot t}))\cong X_r(\overline{\IF_p})_{\sigma}$ in the above notation), and $\Outdiag(S)$, the \emph{outer diagonal automorphism group of $S$}, is the image of $\Inndiag(S)$ under the canonical projection $\Aut(S)\rightarrow\Out(S)$. As in \cite[Theorem 2.5.12(b), p.~58]{GLS98a}, we also view $\Phi_S$ and $\Gamma_S$ as subsets of $\Out(S)$, depending on the context. When $\alpha\in\Aut(S)$ (resp.~$\alpha\in\Out(S)$), then as stated in \cite[p.~105]{Har92a}, $\alpha$ admits a unique factorization into an element of $\Inndiag(S)$ (resp.~$\Outdiag(S)$), an element of $\Phi_S$ and an element of $\Gamma_S$, and we call these the \emph{inner diagonal component} (resp.~\emph{outer diagonal component}), \emph{field component} and \emph{graph component of $\alpha$}, respectively. The product of the field and graph component of $\alpha$ is also called the \emph{graph-field component of $\alpha$}.

When $P$ is some statement, we denote by $\delta_{[P]}$ the \emph{Kronecker delta} value associated with $P$, which is $1$ if $P$ is true and $0$ otherwise. So, for example, in Case 4 in the proof of Lemma \ref{mainLem4}, where $S$ is the simple Lie type group $G_2(p^f)$, the formula $|\Out(S)|=(1+\delta_{[p=3]})f$ means that $|\Out(S)|=2f$ if $p=3$, and $|\Out(S)|=f$ otherwise.

Throughout the paper, we will use the adjective \enquote{semisimple} in two different contexts: On the one hand, by a \emph{semisimple group}, we mean a group with no nontrivial solvable normal subgroups, which for finite groups $G$ is equivalent to the triviality of their solvable radical $\Rad(G)$; for more information on these groups including facts that will be tacitly used throughout the paper, see \cite[pp.~89ff.]{Rob96a} and \cite[Lemma 1.1]{Ros75a} (the latter implies the important fact that for any finite semisimple group $H$, $\N_{\Aut(\Soc(H))}(H)$ is isomorphic to $\Aut(H)$ via its conjugation action on $H$). On the other hand, for a finite simple group of Lie type $S=\leftidx{^t}{X_r(p^{f\cdot t})}$, viewed as a subgroup of a simple Chevalley group $\overline{S}=X_r(\overline{\IF_p})$, a \emph{semisimple automorphism of $S$} is an inner diagonal automorphism of $S$ whose order is coprime to the defining characteristic $p$ of $S$ (and thus can be seen as the conjugation on $S$ by a suitable semisimple element of $\N_{\overline{S}}(S)\leq\overline{S}$.

Finally, we note that we will be using the following asymptotic notations: Let $M$ be a set, let $f,g,\pcal:M\rightarrow\left(0,\infty\right)$ and let $x$ be a variable ranging over $M$. Then we write

\begin{itemize}
\item \enquote{$f(x)\in\O(g(x))$ as $\pcal(x)\to\infty$.} for \enquote{There exist constants $C_1,C_2>0$ such that for all $x\in M$ with $\pcal(x)\geq C_1$, the inequality $f(x)\leq C_2\cdot g(x)$ holds.}.
\item \enquote{$f(x)\in\Omega(g(x))$ as $\pcal(x)\to\infty$.} for \enquote{$g(x)\in\O(f(x))$ as $\pcal(x)\to\infty$.}.
\item \enquote{$f(x)\in\o(g(x))$ as $\pcal(x)\to\infty$.} for \enquote{For every $\epsilon>0$, there is a constant $C=C(\epsilon)>0$ such that for all $x\in M$ with $\pcal(x)\geq C$, the inequality $f(x)\leq \epsilon g(x)$ holds.}
\item \enquote{$f(x)\in\omega(g(x))$ as $\pcal(x)\to\infty$.} for \enquote{$g(x)\in\o(f(x))$ as $\pcal(x)\to\infty$}.
\end{itemize}

\section{Four lemmas for proving Theorem \ref{mainTheo}}\label{sec2}

\subsection{Reduction to semisimple groups with characteristically simple socle}\label{subsec2P1}

We start with the following simple lemma, which allows us to infer that if there exists any finite group $G$ with $\maol(G)\geq\rho$ and having the nonabelian finite simple group $S$ as a composition factor, then actually $\maol(H)\geq\rho$ for some nontrivial finite semisimple group $H$ with socle a power of $S$:

\begin{lemmma}\label{mainLem1}
Let $G$ be a finite group.

\begin{enumerate}
\item If $N$ is a characteristic subgroup of $G$, then $\maol(G/N)\geq\maol(G)$.
\item If $S$ is a nonabelian composition factor of $G$, then $G$ has a characteristic quotient $H$ which is semisimple and satisfies $\Soc(H)\cong S^n$ for some $n\in\IN^+$.
\end{enumerate}
\end{lemmma}

\begin{proof}
For (1): Denote by $\pi$ the canonical projection $G\rightarrow G/N$. Let $O\subseteq G$ be an $\Aut(G)$-orbit such that $|O|=\MAOL(G)$. Since each automorphism of $G$ induces an automorphism of $G/N$, we have that $\pi[O]\subseteq G/N$ is contained in some $\Aut(G/N)$-orbit $\tilde{O}$. Hence

\[
\MAOL(G)=|O|\leq|\tilde{O}|\cdot|N|\leq\MAOL(G/N)\cdot|N|,
\]

and upon dividing both sides of that inequality by $|G|$, we get the desired bound $\maol(G)\leq\maol(G/N)$.

For (2): We recursively define a sequence $(G_m)_{m\in\IN}$ of characteristic quotients of $G$ as follows: $G_0:=G/\Rad(G)$, and $G_{m+1}:=(G_m/\Soc(G_m))/(\Rad(G_m/\Soc(G_m)))$. Then by construction, each of the groups $G_m$, $m\in\IN$, is semisimple, and for all but finitely many $m\in\IN$, $G_m$ is trivial. Furthermore, since solvable finite groups only have abelian composition factors, each nonabelian composition factor of $G$, in particular $S$, is a composition factor (hence a direct factor) of at least one of the groups $\Soc(G_m)$, $m\in\IN$. Say $\Soc(G_m)$ contains $S$ as a composition factor, and let us write $\Soc(G_m)=S^n\times S_1^{n_1}\times\cdots\times S_r^{n_r}$ with $r\in\IN$, $n_1,\ldots,n_r\in\IN^+$, and $S,S_1,\ldots,S_r$ pairwise nonisomorphic nonabelian finite simple groups. Set $H:=G_m/\C_{G_m}(S^n)$. Then the restriction of the canonical projection $G_m\rightarrow H$ to $S^n\leq\Soc(G_m)$ is injective, so we can view $S^n$ embedded as a normal subgroup in $H$. Moreover, by construction, the conjugation action of $H$ on $S^n$ is faithful. Therefore, $H$ is semisimple with socle $S^n$, and since $H$ is a characteristic quotient of $G_m$, $H$ is also a characteristic quotient of $G$.
\end{proof}

\subsection{Conjugacy in wreath products and automorphism orbits in finite semisimple groups}\label{subsec2P2}

Let $G$ be a finite group, let $n\in\IN^+$, and consider an element $\vec{g}=(g_1,\ldots,g_n)\sigma$, with $g_1,\ldots,g_n\in G$ and $\sigma\in\Sym_n$, of the wreath product $G\wr\Sym_n=G^n\rtimes\Sym_n$. In Lemma \ref{mainLem2}(1) below, we give a combinatorial characterization of when another element $\vec{h}=(h_1,\ldots,h_n)\upsilon\in G\wr\Sym_n$ is $(G\wr\Sym_n)$-conjugate to $w$. We note that conjugacy of elements in wreath products has already been studied in \cite{CH06a}, but the focus there was on efficient computation of a set of representatives for the conjugacy classes, and the situation was more complex because the wreath product could be taken with any permutation group (not just $\Sym_n$); in our more special situation, a particularly nice combinatorial characterization of conjugacy can be given.

Lemma \ref{mainLem2}(2) then is essentially an application of Lemma \ref{mainLem2}(1) to the special case $G:=\Aut(S)$, the automorphism group of a nonabelian finite simple group $S$. It gives an upper bound on $\maol(H)$ for $H$ a nontrivial finite semisimple group with $\Soc(H)$ isomorphic to a power of $S$.

Before formulating and proving Lemma \ref{mainLem2}, we need to introduce some notation.

\begin{deffinition}\label{typeDef}
Let $S$ be a nonabelian finite simple group, $\pi:\Aut(S)\rightarrow\Out(S)$ the canonical projection, $c$ an $\Aut(S)$-conjugacy class.

\begin{enumerate}
\item The term \emph{$S$-type} is a synonym for \enquote{$\Out(S)$-conjugacy class}.
\item We call $\pi[c]$, which is an $\Out(S)$-conjugacy class, the \emph{$S$-type of $c$}, denoted by $\type_S(c)$.
\item We set $\rho(c):=\frac{|c|}{|S|\cdot|\type_S(c)|}$. Equivalently, $\rho(c)$ is the proportion of elements of $c$ within any of the $S$-cosets in $\Aut(S)$ which $c$ intersects.
\end{enumerate}
\end{deffinition}

\begin{remmark}\label{typeRem}
Let $S$ be a nonabelian finite simple group. We give two remarks concerning Definition \ref{typeDef}:

\begin{enumerate}
\item For all $\Aut(S)$-conjugacy classes $c$, we have $\rho(c)\leq\h(S)$ (see Notation \ref{hNot}).
\item For all $S$-types $\tau$, the sum of the positive real numbers $\rho(c)$, where $c$ ranges over the $\Aut(S)$-conjugacy classes of $S$-type $\tau$, equals $1$.
\end{enumerate}
\end{remmark}

\begin{deffinition}\label{bcpcDef}
Let $G$ be a finite group, $n\in\IN^+$, and let $w=(g_1,\ldots,g_n)\sigma\in G\wr\Sym_n$.

\begin{enumerate}
\item For a length $l$ cycle $\zeta=(i_1,\ldots,i_l)$ of $\sigma$, we define $\bcpc_{\zeta}(w):=(g_{i_l}g_{i_{l-1}}\cdots g_{i_1})^G$, the \emph{backward cycle product class of $w$ with respect to $\zeta$}, a $G$-conjugacy class.
\item For $l\in\{1,\ldots,n\}$, we denote by $M_l(w)$ the (possibly empty) multiset of the $\bcpc_{\zeta}(w)$, where $\zeta$ runs through the length $l$ cycles of $\sigma$.
\item In the special case where $G=\Aut(S)$ for a nonabelian finite simple group $S$, we denote, for $l\in\{1,\ldots,n\}$ and $\tau$ an $S$-type, by $M_l^{\tau}(w)$ the (possibly empty) multiset of the $\bcpc_{\zeta}(w)$, where $\zeta$ runs through the length $l$ cycles of $\sigma$ such that $\type_S(\bcpc_{\zeta}(w))=\tau$.
\end{enumerate}
\end{deffinition}

\begin{deffinition}\label{rDef}
Let $S$ be a nonabelian finite simple group, and let $M$ be a nonempty multiset of $\Aut(S)$-conjugacy classes of a common $S$-type $\tau$.

\begin{enumerate}
\item We set $n(M):=|M|$, the multiset-cardinality of $M$.
\item We define $k(\tau)$ as the number of $\Aut(S)$-conjugacy classes of type $\tau$.
\item For an $\Aut(S)$-conjugacy class $c$, we define $l_c(M)$ as the multiplicity with which $c$ occurs as an element of $M$ (of course, $l_c(M)=0$ unless $c$ is of $S$-type $\tau$).
\item Let $c_1^{(\tau)},\ldots,c_{k(\tau)}^{(\tau)}$ be the distinct $\Aut(S)$-conjugacy classes of $S$-type $\tau$. We define

\[
r(M):={n(M) \choose l_{c_1^{(\tau)}}(M),\ldots,l_{c_{k(\tau)}^{(\tau)}}(M)}\cdot\rho(c_1^{(\tau)})^{l_{c_1^{(\tau)}}(M)}\cdots\rho(c_{k(\tau)}^{(\tau)})^{l_{c_k^{(\tau)}}(M)}.
\]
\end{enumerate}
\end{deffinition}

Note that in the situation of Definition \ref{rDef}(4), $r(M)$ is the evaluation at the argument $(l_{c_1^{(\tau)}}(M),\ldots,l_{c_{k(\tau)}^{(\tau)}}(M))$ of the probability mass function of the multinomial distribution with parameters $n(M)$, the number of experiment repetitions, and $(\rho(c_1^{(\tau)}),\ldots,\rho(c_{k(\tau)}^{(\tau)}))$, the ordered list of success probabilities of the outcomes of a single experiment.

\begin{lemmma}\label{mainLem2}
The following hold:

\begin{enumerate}
\item Let $G$ be a finite group, $n\in\IN^+$, $g_1,\ldots,g_n,h_1,\ldots,h_n\in G$ and $\sigma,\upsilon\in\Sym_n$. The following are equivalent:

\begin{enumerate}
\item The two elements $\vec{g}:=(g_1,\ldots,g_n)\sigma$ and $\vec{h}:=(h_1,\ldots,h_n)\upsilon$ of $G\wr\Sym_n$ are conjugate.
\item $\sigma$ and $\upsilon$ have the same cycle type (i.e., are conjugate in $\Sym_n$), and for all $l\in\{1,\ldots,n\}$, the equality of multisets $M_l(\vec{g})=M_l(\vec{h})$ holds.
\end{enumerate}

\item Let $S$ be a nonabelian finite simple group, $n\in\IN^+$, let $H$ be a finite group such that $S^n\leq H\leq\Aut(S^n)=\Aut(S)\wr\Sym_n$, and let $\vec{\alpha}:=(\alpha_1,\ldots,\alpha_n)\sigma\in H$. Then

\[
\frac{1}{|H|}\cdot|\vec{\alpha}^{\Aut(H)}|\leq\prod_l\prod_{\tau}{r(M_l^{\tau}(\vec{\alpha}))},
\]

where $l$ runs through the cycle lengths of $\sigma$ and $\tau$ runs through those $\Out(S)$-conjugacy classes that occur as the $S$-type of one of the $\bcpc_{\zeta}(\vec{\alpha})$ where $\zeta$ is a length $l$ cycle of $\sigma$.
\end{enumerate}
\end{lemmma}

\begin{proof}
For (1): Fix $\vec{g}=(g_1,\ldots,g_n)\sigma\in G\wr\Sym_n$. It is clear that the possible permutation parts of the $(G\wr\Sym_n)$-conjugates of $\vec{g}$ are just the $\Sym_n$-conjugates of $\sigma$, and since for each cycle $\zeta$ of $\sigma$ and all $\psi\in\Sym_n$, $\bcpc_{\zeta}(\vec{g})=\bcpc_{\zeta^{\psi}}(\vec{g}^{\psi})$, we have $M_l(\vec{g})=M_l(\vec{g}^{\psi})=M_l((g_1,\ldots,g_n)^{\psi}\sigma^{\psi})$ for all $l\in\{1,\ldots,n\}$ and all $\psi\in\Sym_n$. Therefore, statement (1) follows from the following special case of it, which is a characterization of the conjugates of $\vec{g}$ that have the same permutation part as $\vec{g}$:

\enquote{An element $\vec{h}=(h_1,\ldots,h_n)\sigma\in G\wr\Sym_n$ is $(G\wr\Sym_n)$-conjugate to $\vec{g}$ if and only if for all $l\in\{1,\ldots,n\}$, $M_l(\vec{g})=M_l(\vec{h})$.}

We will show this assertion, so we need to investigate which $\vec{h}=(h_1,\ldots,h_n)\sigma\in G\wr\Sym_n$ are conjugate to $\vec{g}$. Consider a conjugator $\vec{k}=(k_1,\ldots,k_n)\psi\in G\wr\Sym_n$. The permutation part of $\vec{g}^{\vec{k}}$ is equal to $\sigma$ if and only if $\psi\in\C_{\Sym_n}(\sigma)$, so we assume this from now on. For $l=1,\ldots,n$, we denote by $u_l$ the number of $l$-cycles of $\sigma$, and by $I_l$ the subset of $\{1,\ldots,n\}$ consisting of those indices lying on an $l$-cycle of $\sigma$. Note that $|I_l|=l\cdot u_l$, and that $\C_{\Sym_n}(\sigma)$ is the permutational direct product of the subgroups $C_l\leq\Sym_{I_l}$, $l=1,\ldots,n$, where $C_l$ is the permutational wreath product of the cyclic degree $l$ permutation group generated by an $l$-cycle with $\Sym_{u_l}$, where $\Sym_{u_l}$ acts on the supports of the $u_l$ cycles of length $l$ of $\sigma$. Therefore, the possible values of $\vec{g}^{\psi}=(g_{\psi^{-1}(1)},\ldots,g_{\psi^{-1}(n)})\sigma=(g'_1,\ldots,g'_n)\sigma=:\vec{g'}$ are just those where the $n$-tuple $(g'_1,\ldots,g'_n)$ is obtained from $(g_1,\ldots,g_n)$ by permuting entry groups that are supports of $\sigma$-cycles of the same length and cyclically shifting (with respect to the order of indices on the corresponding $\sigma$-cycle) such entry groups. In particular, we have the following:

\begin{enumerate}
\item For each $\psi\in\C_{\Sym_n}(\sigma)$, we have that for all $l=1,\ldots,n$, $M_l(\vec{g})=M_l(\vec{g}^{\psi})$, or equivalently, that there is a multiset bijection $\chi_{\psi}$ from the set $\{\zeta_1^{(l)},\ldots,\zeta_{u_l}^{(l)}\}$ of length $l$ cycles of $\sigma$ to the multiset $M_l(\vec{g})$ such that for $i=1,\ldots,u_l$, $\bcpc_{\zeta_i^{(l)}}(\vec{g}^{\psi})=\chi_{\psi}(\zeta_i^{(l)})$.
\item For every multiset bijection $\chi:\{\zeta_1^{(l)},\ldots,\zeta_{u_l}^{(l)}\}\rightarrow M_l(\vec{g})$, there is a $\psi\in\C_{\Sym_n}(\sigma)$ such that $\chi_{\psi}=\chi$.
\end{enumerate}

Next, we study which elements $\vec{h}=(h_1,\ldots,h_n)\sigma$ of $G\wr\Sym_n$ one can reach from $\vec{g}=(g_1,\ldots,g_n)\sigma$ via conjugation by an $n$-tuple $(k_1,\ldots,k_n)\in G^n$. Similar considerations were already used by Cannon and Holt in \cite[proof of Proposition 2.1]{CH06a}, but for the reader's convenience, we give a more elaborate form of their argument here.

First, observe that for each cycle $\zeta$ of $\sigma$, those entries of the tuple part of $\vec{g}^{(k_1,\ldots,k_n)}$ that correspond to indices from the support of $\zeta$ only depend on such entries from $\vec{g}$ and from $(k_1,\ldots,k_n)$, so we can study this problem \enquote{cycle-wise}. Hence let $\zeta=(i_1,\ldots,i_l)$ be a cycle of $\sigma$, and assume w.l.o.g.~that $i_1<i_2<\cdots<i_l$. Then, using the symbol $\ast$ as a substitute for (possibly empty) entry groups outside the support of $\zeta$, we have that

\begin{align*}
\vec{g}^{(k_1,\ldots,k_n)} &=((\ast,g_{i_1},\ast,g_{i_2},\ast,\ldots,\ast,g_{i_l},\ast)\sigma)^{(\ast,k_{i_1},\ast,k_{i_2},\ast,\ldots,\ast,k_{i_l},\ast)} \\
&=(\ast,k_{i_1}g_{i_1}k_{\sigma^{-1}(i_1)}^{-1},\ast,k_{i_2}g_{i_2}k_{\sigma^{-1}(i_2)}^{-1},\ast,\ldots,\ast,k_{i_l}g_{i_l}k_{\sigma^{-1}(i_l)}^{-1},\ast)\sigma \\
&=(\ast,k_{i_1}g_{i_1}k_{i_m}^{-1},\ast,k_{i_2}g_{i_2}k_{i_1}^{-1},\ast,\ldots,\ast,k_{i_l}g_{i_l}k_{i_{l-1}}^{-1},\ast)\sigma.
\end{align*}

Now fix $(h_1,\ldots,h_n)\in G^n$ and set the last expression in the above chain of equalities equal to $\vec{h}=(h_1,\ldots,h_n)\sigma$. Focussing again only on the entries corresponding to indices lying on the support of $\zeta$, we get the following system of $l$ equations over $G$ in the variables $k_{i_1},\ldots,k_{i_l}$:

\begin{align*}
k_{i_1}g_{i_1}k_{i_m}^{-1} &= h_{i_1}, \\
k_{i_2}g_{i_2}k_{i_1}^{-1} &= h_{i_2}, \\
&\vdots \\
k_{i_l}g_{i_l}k_{i_{l-1}}^{-1} &= h_{i_l}.
\end{align*}

By isolating the variable $k_{i_j}$ in the $j$-th equation and making a chain of substitutions, we find that the solvability over $G$ of this system of equations is equivalent to the solvability over $G$ of the following single equation in the variable $k_{i_l}$:

\[
k_{i_l}=h_{i_l}h_{i_{l-1}}\cdots h_{i_1}k_{i_l}g_{i_1}^{-1}g_{i_2}^{-1}\cdots g_{i_l}^{-1}.
\]

Therefore, the system of equations associated with the $\sigma$-cycle $\zeta$ is solvable over $G$ if and only if $\bcpc_{\zeta}(\vec{h})=\bcpc_{\zeta}(\vec{g})$, and in summary, the following hold:

\begin{enumerate}
\item For each $(k_1,\ldots,k_n)\in G^n$, we have that for all cycles $\zeta$ of $\sigma$, $\bcpc_{\zeta}(\vec{g})=\bcpc_{\zeta}(\vec{g}^{(k_1,\ldots,k_n)})$.
\item Conversely, for every $\vec{h}=(h_1,\ldots,h_n)\sigma\in G\wr\Sym_n$ such that for every cycle $\zeta$ of $\sigma$, $\bcpc_{\zeta}(\vec{g})=\bcpc_{\zeta}(\vec{h})$, there is a $(k_1,\ldots,k_n)\in G^n$ such that $\vec{g}^{(k_1,\ldots,k_n)}=\vec{h}$.
\end{enumerate}

By combining these insights on the possible values of $\vec{g}^{\psi}$ and $\vec{g}^{(k_1,\ldots,k_n)}$, one gets the asserted result.

For (2): Set $K:=\Aut(S)^n\cap H$. Then $K$ is a characteristic subgroup of $H$, and the $K$-cosets in $H$ correspond to the various permutation parts of elements of $H$. Since $\vec{\alpha}^{\Aut(H)}$ distributes evenly among the $K$-cosets which it intersects, it suffices to show that the proportion of $\vec{\alpha}^{\Aut(H)}\cap K\vec{\alpha}$ within $K\vec{\alpha}$ is at most the indicated product of $r$-values. Now $K\vec{\alpha}$ is itself a disjoint union of $S^n$-cosets in $H$, each of the form $S^n\vec{\beta}=S^n(\beta_1,\ldots,\beta_n)\sigma=(S\beta_1,\ldots,S\beta_n)\sigma$ for some fixed $\beta_1,\ldots,\beta_n\in\Aut(S)$, and so it suffices to show that the proportion of $\vec{\alpha}^{\Aut(H)}\cap S^n\vec{\beta}$ in $S^n\vec{\beta}$ is at most the indicated product. We will actually show the stronger statement that the proportion of $\vec{\alpha}^{\Aut(S^n)}\cap S^n\vec{\beta}$ in $S^n\vec{\beta}$ is at most the indicated product.

For $(s_1,\ldots,s_n)\in S^n$, set $\vec{\beta}(s_1,\ldots,s_n):=(s_1\beta_1,\ldots,s_n\beta_n)\sigma$. We want to bound the proportion of $(s_1,\ldots,s_n)\in S^n$ such that $\vec{\beta}(s_1,\ldots,s_n)\in\vec{\alpha}^{\Aut(S^n)}$, which by statement (1) is equivalent to $M_l^{\tau}(\vec{\alpha})=M_l^{\tau}(\vec{\beta}(s_1,\ldots,s_n))$ for all $l\in\{1,\ldots,n\}$ and all $S$-types $\tau$, or equivalently, just for those $(l,\tau)$ such that $M_l^{\tau}(\vec{\alpha})$ is nonempty.

Observe that for each cycle $\zeta$ of $\sigma$, the $S$-type $\tau_{\zeta}$ of $\bcpc_{\zeta}(\vec{\beta}(s_1,\ldots,s_n))$ is independent of $(s_1,\ldots,s_n)$. In what follows, $l$ is always an element of $\{1,\ldots,n\}$, and $\tau$ is an $S$-type. We denote by $u_{l,\tau}$ the number of $l$-cycles $\zeta$ of $\sigma$ such that $\tau_{\zeta}=\tau$. Now if, for some $(l,\tau)$, we have $u_{l,\tau}\not=|M_l^{\tau}(\vec{\alpha})|=n(M_l^{\tau}(\vec{\alpha}))$, then the intersection $\vec{\alpha}^{\Aut(S^n)}\cap S^n\vec{\beta}(1_S,\ldots,1_S)$ is empty, and so we may henceforth assume that $u_{l,\tau}=n(M_l^{\tau}(\vec{\alpha}))$ for all $(l,\tau)$, and under this assumption, we will actually show that the proportion of $\vec{\alpha}^{\Aut(S^n)}\cap S^n\vec{\beta}$ in $S^n\vec{\beta}$ is equal to the indicated product of $r$-values.

For each $(l,\tau)$, the single condition $M_l^{\tau}(\vec{\beta}(s_1,\ldots,s_n))=M_l^{\tau}(\vec{\alpha})$ imposes restrictions only on those variables $s_i$ where $i$ is an element of the support of one of the length $l$ cycles $\zeta$ of $\sigma$ such that $\tau_{\zeta}=\tau$, and so for distinct $(l,\tau)$, the corresponding conditions concern disjoint variable sets. It therefore suffices to show that for fixed $(l,\tau)$, the proportion of $(s_1,\ldots,s_n)\in S^n$ such that $M_l^{\tau}(\vec{\beta}(s_1,\ldots,s_n))=M_l^{\tau}(\vec{\alpha})$ is exactly $r(M_l^{\tau}(\vec{\alpha}))$.

Let $\zeta_{l,\tau}^{(1)}=(i_{l,\tau}^{(1,1)},\ldots,i_{l,\tau}^{(1,l)}),\ldots,\zeta_{l,\tau}^{u_{l,\tau}}=(i_{l,\tau}^{(u_{l,\tau},1)},\ldots,i_{l,\tau}^{(u_{l,\tau},l)})$ be the $u_{l,\tau}$ distinct $l$-cycles $\zeta$ of $\sigma$ such that $\tau_{\zeta}=\tau$. We show that for all $(s_i)_{i\in\{1,\ldots,n\}\setminus\{i_{l,\tau}^{(1,l)},\ldots,i_{l,\tau}^{(u_{l,\tau},l)}\}}\in S^{n-u_{l,\tau}}$, the proportion of $(s_{i_{l,\tau}^{(1,l)}},\ldots,s_{i_{l,\tau}^{(u_{l,\tau,l)}}})\in S^{u_{l,\tau}}$ such that $M_l^{\tau}(\vec{\beta}(s_1,\ldots,s_n))=M_l^{\tau}(\vec{\alpha})$ is exactly $r(M_l^{\tau}(\vec{\alpha}))$.

As in Definition \ref{rDef}(4), let $c_1^{(\tau)},\ldots,c_{k(\tau)}^{(\tau)}$ be the $k(\tau)$ distinct $\Aut(S)$-conjugacy classes of $S$-type $\tau$. Now $M_l^{\tau}(\vec{\beta}(s_1,\ldots,s_n))=M_l^{\tau}(\vec{\alpha})$ just means that among the $u_{l,\tau}$ many elements $\xi_j(s_{i_{l,\tau}^{(j,l)}}):=s_{i_{l,\tau}^{(j,l)}}\beta_{i_{l,\tau}^{(j,l)}}s_{i_{l,\tau}^{(j,l-1)}}\beta_{i_{l,\tau}^{(j,l-1)}}\cdots s_{i_{l,\tau}^{(j,1)}}\beta_{i_{l,\tau}^{(j,1)}}\in\Aut(S)$ for $j=1,\ldots,u_{l,\tau}$, precisely $l_{c_t^{(\tau)}}(M_l^{\tau}(\vec{\alpha}))$ many lie in $c_t^{(\tau)}$ for each $t\in\{1,\ldots,k(\tau)\}$. Consider drawing the $s_{i_{l,\tau}^{(j,l)}}$, $j=1,\ldots,u_{l,\tau}$, from $S$ at random. Then for each $j\in\{1,\ldots,u_{l,\tau}\}$ and each $t\in\{1,\ldots,k(\tau)\}$, the probability that $\xi_j(s_{i_{l,\tau}^{(j,l)}})$ lands in $c_t^{(\tau)}$ is precisely $\rho(c_t^{(\tau)})$. Therefore, the probability that $M_l^{\tau}(\vec{\beta}(s_1,\ldots,s_n))=M_l^{\tau}(\vec{\alpha})$ is just the multinomial distribution probability mass function value $r(M_l^{\tau}(\vec{\alpha}))$, as required.
\end{proof}

\subsection{A lemma on the multinomial distribution}\label{subsec2P3}

The next ingredient in the proof of our main results is a method to bound the factors $r(M_l^{\tau}(\vec{\alpha}))$ in Lemma \ref{mainLem2}(2), which is provided by the following lemma, essentially reducing the problem to bounding the numbers $\rho(c)$ introduced in Definition \ref{typeDef}(3):

\begin{lemmma}\label{mainLem3}
Let $k\in\IN^+$, let $\rho_1,\ldots,\rho_k\in\left[0,1\right]$ such that $\rho_1+\cdots+\rho_k=1$, and let $n\in\IN^+$. Then the following hold:

\begin{enumerate}
\item Let $l_1,\ldots,l_k\in\IN^+$ with $l_1\geq l_2\geq\cdots\geq l_k$ and $n=l_1+\cdots+l_k$. Then unless $n=k\in\{2,3\}$, we have ${n \choose l_1,\ldots,l_k}\rho_1^{l_1}\cdots\rho_k^{l_k}\leq\rho_1$.
\item Let $l_1,\ldots,l_k\in\IN$ with $n=l_1+\cdots+l_k$. Then ${n \choose l_1,\ldots,l_k}\rho_1^{l_1}\cdots\rho_k^{l_k}\leq\max_{i=1,\ldots,k}{\rho_i}$.
\end{enumerate}
\end{lemmma}

Lemma \ref{mainLem3}(2) can be equivalently reformulated as follows: The values of the probability mass function of a multinomial distribution are all bounded from above by the maximum success probability of an outcome of the random experiment that is repeated.

\begin{proof}
For (1): Consider the function

\[
f_{l_1,\ldots,l_k}:\left[0,1\right]^k\rightarrow\IR, (x_1,\ldots,x_k)\mapsto{n \choose l_1,\ldots,l_k}x_1^{l_1-1}x_2^{l_2}\cdots x_k^{l_k},
\]

and note that we are done if we can show that unless $n=k\in\{2,3\}$, all values of $f_{l_1,\ldots,l_k}$ on arguments $(x_1,\ldots,x_k)$ with $x_1+\cdots+x_k=1$ are bounded from above by $1$. This is clear for $k=1$, so we may assume $k\geq 2$ throughout the rest of the proof of statement (1). The method of Lagrange multipliers yields that the maximum value of $f_{l_1,\ldots,l_k}$ on arguments whose entries sum up to $1$ is attained at $(\frac{l_1-1}{n-1},\frac{l_2}{n-1},\ldots,\frac{l_k}{n-1})$, and so to verify the inequality for a particular choice of $n$ and $l_1,\ldots,l_k$, it suffices to check that $f_{l_1,\ldots,l_k}(\frac{l_1-1}{n-1},\frac{l_2}{n-1},\ldots,\frac{l_k}{n-1})\leq 1$. Using a computer, one can thus verify the asserted inequality for

\begin{itemize}
\item $n\in\{1,\ldots,9\}$ and $k=4$,
\item $n\in\{1,\ldots,15\}\setminus\{3\}$ and $k=3$,
\item $n\in\{10,\ldots,96\}\setminus\{2\}$ and $k=2$.
\end{itemize}

We may thus henceforth assume that $(n,k)$ is none of the above listed pairs checked by computer. Using Robbins's explicit Stirling-type bounds for the factorial \cite{Rob55a}, we get that

\begin{align*}
{n \choose l_1,\ldots,l_k}\rho_1^{l_1}\cdots\rho_k^{l_k}&=\frac{n!}{l_1!\cdots l_k!}\rho_1^{l_1}\cdots\rho_k^{l_k}\leq\frac{\sqrt{2\pi n}(n/\e)^n\e^{1/(12n)}}{\prod_{i=1}^k{(\sqrt{2\pi l_i}(l_i/\e)^{l_i}\e^{1/(12l_i+1)})}}\rho_1^{l_1}\cdots\rho_k^{l_k} \\
&\leq\frac{1}{(2\pi)^{(k-1)/2}}\cdot\e^{1/(12n)}\cdot\sqrt{\frac{n}{l_1\cdots l_k}}\cdot(\frac{n\rho_1}{l_1})^{l_1}\cdots(\frac{n\rho_k}{l_k})^{l_k} \\
&\leq\frac{1}{(2\pi)^{(k-1)/2}}\cdot\e^{1/12}\cdot\sqrt{\frac{n}{l_1\cdots l_k}}\cdot\frac{n}{l_1}\rho_1\cdot(\frac{n}{n-1})^{n-1} \\
&\leq\frac{n/l_1}{(2\pi)^{(k-1)/2}}\cdot\e^{1/12}\cdot\sqrt{\frac{n}{l_1\cdots l_k}}\cdot\e\cdot\rho_1 \\
&=\frac{n/l_1}{(2\pi)^{(k-1)/2}}\cdot\e^{13/12}\cdot\sqrt{\frac{n}{l_1\cdots l_k}}\cdot\rho_1,
\end{align*}

where in passing from the second to the third line, the inequality of the arithmetic and geometric means was applied. Therefore, we are done if we can show that

\begin{equation}\label{eq1}
\frac{n/l_1}{(2\pi)^{(k-1)/2}}\cdot\e^{13/12}\cdot\sqrt{\frac{n}{l_1\cdots l_k}}\leq 1
\end{equation}

for all $n$ and $k$ (and associated choices of $l_1,\ldots,l_k$) other than those checked with a computer above. But since $l_1$ is largest among the $k$ numbers $l_i$ summing up to $n$, we have $n/l_1\leq k$, and so for each choice of $n$, $k$ and $l_1,\ldots,l_k$, Formula (\ref{eq1}) is implied by

\begin{equation}\label{eq2}
\frac{k}{(2\pi)^{(k-1)/2}}\cdot\e^{13/12}\cdot\sqrt{\frac{n}{l_1\cdots l_k}}\leq 1.
\end{equation}

We will first deal with the cases $k\in\{2,3,4\}$ separately before giving a uniform argument for $k\geq 5$.

For $k=2$, where we may assume $n\geq97$, let us first assume that $l_2=1$. We will then show directly that $n\rho_1^{n-1}\rho_2\leq\rho_1$, or equivalently, $n\rho_1^{n-2}\rho_2\leq 1$. By the Lagrange multiplier method, this holds if $n\cdot(\frac{n-2}{n-1})^{n-2}\cdot\frac{1}{n-1}\leq 1$, which is equivalent to $(1-\frac{1}{n-1})^{n-1}\leq 1-\frac{2}{n}$, which actually holds for all $n\geq 4$, as the left-hand side is bounded from above by $\e^{-1}=0.3678\ldots$. Let us next assume that $l_2\geq 12$, in which case we will be able to verify the validity of Formula (\ref{eq2}), as follows: For $k=2$, Formula (\ref{eq2}) is equivalent to $\frac{n}{l_1l_2}\leq\frac{\pi}{2}\e^{-13/6}$, and since $n/l_1\leq 2$, Formula (\ref{eq2}) is therefore implied by $l_2\geq\frac{4}{\pi}\e^{13/6}=11.1142\ldots$. So in our proof for $k=2$, we may henceforth assume that $l_2\in\{2,\ldots,11\}$, and under this assumption, we will verify the validity of Formula (\ref{eq1}). Since $l_1\geq n-11$ and $l_2\geq 2$, Formula (\ref{eq1}) is now implied by $\frac{1}{\sqrt{2\pi}}\frac{n}{n-11}\e^{13/12}\sqrt{\frac{n}{2(n-11)}}\leq1$, or equivalently, $1+\frac{11}{n-11}\leq(\sqrt{4\pi}\e^{-13/12})^{2/3}=1.1291\ldots$, which holds for $n\geq97$, as required.

For $k=3$, where we may assume $n\geq16$, let us first assume that $l_2l_3\geq6$. Under this assumption, we will verify the validity of Formula (\ref{eq2}), which here is equivalent to $\frac{n}{l_1l_2l_3}\leq\frac{4\pi^2}{9}\e^{-13/6}$, or $l_2l_3\geq\frac{n}{l_1}\cdot\frac{9}{4\pi^2}\e^{13/6}$. But $\frac{n}{l_1}\leq 3$, and $3\cdot\frac{9}{4\pi^2}\e^{13/6}=5.9700\ldots$. We may thus assume $l_2l_3\leq 5$ now, so that $l_2+l_3\leq6$. Then $l_1\geq n-6$, and hence Formula (\ref{eq1}) is implied by $\frac{1}{2\pi}\frac{n}{n-6}\e^{13/12}\sqrt{n/(n-6)}\leq 1$, or equivalently, $1+\frac{6}{n-6}\leq(2\pi\e^{-13/12})^{2/3}=1.6537\ldots$, which holds for $n\geq16$, as required.

For $k=4$, we may assume $n\geq10$. Just as for $k=3$, one checks that Formula (\ref{eq2}) is now implied by $l_2l_3l_4\geq\frac{8}{\pi^3}\e^{13/6}=2.2522\ldots$, so we may assume $l_2l_3l_4\leq2$, so that $l_2+l_3+l_4\leq 4$ and $l_1\geq n-4$. But Formula (\ref{eq2}) for $k=4$ is equivalent to $\frac{n}{l_1l_2l_3l_4}\leq\frac{\pi^3}{2}\e^{-13/6}=1.7760\ldots$, which is implied by $1+\frac{4}{n-4}\leq1.7760\ldots$, and this holds for $n\geq10$, as required.

Let us now assume $k\geq 5$. Using once more that $n/l_1\leq k$, we find that Formula (\ref{eq2}) is implied by $\frac{k}{l_2\ldots l_k}\leq\frac{(2\pi)^{k-1}}{k^2}\e^{-13/6}$, which in turn is a consequence of $\e^{13/6}k^3\leq(2\pi)^{k-1}$, which can be verified to hold for all $k\geq 5$ by an inductive argument.

For (2): Let us first argue why we may assume that all $l_i$ are positive. Indeed, assume we know statement (2) to be true if all $l_i$ are positive, and say w.l.o.g.~$l_1,\ldots,l_t\geq 1$ and $l_{t+1}=l_{t+2}=\cdots=l_k=0$ for some $t\in\{1,\ldots,k-1\}$. Set

\[
C:={n\choose l_1,\ldots,l_k}\rho_1^{l_1}\cdots\rho_k^{l_k}={n\choose l_1,\ldots,l_t}\rho_1^{l_1}\cdots\rho_t^{l_t}.
\]

Then

\begin{align*}
\frac{1}{(\rho_1+\cdots+\rho_t)^n}C &={n\choose l_1,\ldots,l_t}(\frac{1}{\rho_1+\cdots+\rho_t}\rho_1)^{l_1}\cdots(\frac{1}{\rho_1+\cdots+\rho_t}\rho_t)^{l_t} \\
&\leq\max_{i=1,\ldots,t}{\frac{1}{\rho_1+\cdots+\rho_t}\rho_i}\leq\frac{1}{\rho_1+\cdots+\rho_t}\max_{i=1,\ldots,k}{\rho_i},
\end{align*}

so that

\[
C\leq(\rho_1+\cdots+\rho_t)^{n-1}\max_{i=1,\ldots,k}{\rho_i}\leq\max_{i=1,\ldots,k}{\rho_i},
\]

as required.

Let us now prove statement (2) under the assumption that all $l_i$ are positive. Unless $n=k\in\{2,3\}$, this is a direct consequence of statement (1), so let us go through these final two cases:

For $n=k=2$, we need to check that for all $(\rho_1,\rho_2)\in\left[0,1\right]^2$ with $\rho_1+\rho_2=1$, we have $2\rho_1\rho_2\leq\max\{\rho_1,\rho_2\}$. This holds because at least one of $\rho_1,\rho_2$ is at most $\frac{1}{2}$.

For $n=k=3$, we need to check that for all $(\rho_1,\rho_2,\rho_3)\in\left[0,1\right]^3$ with $\rho_1+\rho_2+\rho_3=1$, we have $6\rho_1\rho_2\rho_3\leq\max\{\rho_1,\rho_2,\rho_3\}$. This is clear if at least two of the three numbers $\rho_1,\rho_2,\rho_3$ are at most $\frac{1}{\sqrt{6}}$, so assume w.l.o.g.~that $\rho_1,\rho_2>\frac{1}{\sqrt{6}}$. Then $\rho_3<1-\frac{2}{\sqrt{6}}$, and so, since at least one of $\rho_1,\rho_2$ is at most $\frac{1}{2}$, say w.l.o.g.~$\rho_2\leq\frac{1}{2}$, we get that $6\rho_1\rho_2\rho_3\leq6\cdot\frac{1}{2}\cdot(1-\frac{2}{\sqrt{6}})\cdot\rho_1=0.5505\ldots\rho_1\leq\rho_1\leq\max\{\rho_1,\rho_2,\rho_3\}$, as required.
\end{proof}

\subsection{Conjugacy classes in automorphism groups of nonabelian finite simple groups}\label{subsec2P4}

By Lemma \ref{mainLem3}(2), each of the factors $r(M_l^{\tau}(\vec{\alpha}))$ in Lemma \ref{mainLem2}(2) is bounded from above by the maximum value of $\rho(c)$, where $c$ ranges over the conjugacy classes in $\Aut(S)$. Let us introduce a notation for this maximum value:

\begin{nottation}\label{hNot}
Let $S$ be a nonabelian finite simple group. We set

\[
\h(S):=\frac{1}{|S|}\max_{\alpha,\beta\in\Aut(S)}{|\alpha^{\Aut(S)}\cap S\beta|}.
\]
\end{nottation}

A possible strategy to obtain a nontrivial constant upper bound on $\maol(G)$ for finite nonsolvable groups $G$ is to give such a bound on $\h(S)$ for nonabelian finite simple groups $S$, which we do in statement (1) of the following lemma:

\begin{lemmma}\label{mainLem4}
Let $S$ be a nonabelian finite simple group. Then the following hold:

\begin{enumerate}
\item $\h(S)\leq\frac{18}{19}$.
\item There is a function $f:\left(0,1\right]\rightarrow\left(0,\infty\right)$ such that for all $\rho\in\left(0,1\right]$: If $\h(S)\geq\rho$, then

\begin{enumerate}
\item if $S$ is alternating, then the degree of $S$ is at most $f(\rho)$,
\item if $S$ is of Lie type, then the untwisted Lie rank and the defining characteristic of $S$ are at most $f(\rho)$.
\end{enumerate}
\end{enumerate}
\end{lemmma}

Note that in Lemma \ref{mainLem4}(2) and in case $S=\leftidx{^{t}}X_r(p^{f\cdot t})$ is of Lie type, we are \emph{not} asserting that the \enquote{field extension parameter} $f$ is bounded in terms of $\rho$, which is actually not true. As a consequence, while Theorem \ref{mainTheo}(1) follows more or less directly from Lemma \ref{mainLem4}(1) and the previous three lemmas, Theorem \ref{mainTheo}(2) will still require some additional ideas to prove.

As mentioned in the Overview (Subsection \ref{subsec1P2}), the proof of Lemma \ref{mainLem4}(1) will use the classification of the finite simple groups, and particularly for the Lie type case, it is quite laborious. Before starting with the actual proof, we review the basic proof idea and some results from the literature that will be frequently used in the proof.

The basic proof strategy for Lemma \ref{mainLem4} is to derive, for a nonabelian finite simple group $S$, sufficiently good uniform lower bounds on the sizes of element centralizers in $\Aut(S)$, i.e., lower bounds on $\MCS(\Aut(S))$. For example, if we can show that $\MCS(\Aut(S))\geq\frac{19}{18}|\Out(S)|$, then actually all conjugacy classes in $\Aut(S)$ are of length at most $\frac{18}{19}|S|$, so that in particular, $\h(S)\leq\frac{18}{19}$. Unfortunately, this does not always hold (there actually are nonabelian finite simple groups $S$ such that some $\Aut(S)$-conjugacy classes are of length larger than $|S|$, for example $S=A_2(4)=\PSL_3(4)$, which is of order $20160$ and whose automorphism group has a conjugacy class of length $24192$), so sometimes, one needs more sophisticated arguments studying how many $S$-cosets in $\Aut(S)$ a given $\Aut(S)$-conjugacy class intersects (which then breaks down the number of elements per $S$-coset by a certain factor).

For sporadic and alternating groups, this approach works without difficulties and without the need to use results from the literature except, of course, the lexicographical information on sporadic groups (and on $\Aut(\Alt_6)$) from the ATLAS of Finite Group Representations \cite{ATLAS}. For Lie type groups, on the other hand, we will heavily rely on three types of known results:

\begin{itemize}
\item Hartley's descriptions of automorphism centralizers from \cite[Propositions 4.1 and 4.2]{Har92a},
\item Fulman-Guralnick's bounds on the orders of centralizers of inner diagonal automorphisms of classical and Chevalley groups from \cite[Section 6]{FG12a}, and
\item some general techniques for deriving lower bounds on the orders of centralizers of inner diagonal automorphisms of finite simple groups of Lie type developed by Hartley and Kuzucuo\u{g}lu in \cite[proof of Theorem A1, pp.~319f.]{HK91a}.
\end{itemize}

As for Hartley's results, we will require a more detailed version of \cite[Propositions 4.1 and 4.2]{Har92a}, which we give now and which follows by studying the proofs of these results in \cite{Har92a}:

\begin{propposition}\label{hartleyProp}
Let $S$ be a finite simple group of Lie type with defining characteristic $p$, so that $S=O^{p'}(\overline{S}_{\sigma})$ for some simple linear algebraic group $\overline{S}$ over $\overline{\IF_p}$ of adjoint type and a Lang-Steinberg map $\sigma$ on $\overline{S}$. Write $\sigma=\nu\gamma$, where $\nu$ is a Frobenius map (\enquote{Frobenius map in $\Phi$} in the terminology and notation of \cite[Section 4, p.~110]{Har92a}) on $\overline{S}$ and $\gamma$ is a (possibly trivial) graph automorphism of $\overline{S}$ (i.e., an element of $\Gamma$ in the notation of \cite[Section 4, p.~110]{Har92a}).

First, assume additionally that $\overline{S}$ is not of type $B_2$, $F_4$ or $G_2$. Then let $\alpha=s\phi\delta$ be an automorphism of $S$ with $s\in\overline{S}_{\sigma}$, $\phi$ a field automorphism of $S$ and $\delta$ a graph automorphism of $S$. Set $g:=\ord(\phi)$. Then the following hold:

\begin{enumerate}
\item If $\gamma=1$ (i.e., $t(\sigma)=1$), and either $\delta=1$, or $\delta\not=1$ and $\ord(\delta)\nmid g$, then there is another Lang-Steinberg map $\mu$ on $\overline{S}$ such that $\mu(\alpha^g)=\alpha^g$, $\C_{\overline{S}_{\sigma}}(\alpha)=\C_{\overline{S}_{\mu}}(\alpha^g)$, $q(\mu)^g=q(\sigma)$ and $t(\mu)=1$.
\item If $\gamma=1$ (i.e., $t(\sigma)=1$) and $\delta\not=1$ and $\ord(\delta)\mid g$, then there is another Lang-Steinberg map $\mu$ on $\overline{S}$ such that $\mu(\alpha^g)=\alpha^g$, $\C_{\overline{S}_{\sigma}}(\alpha)=\C_{\overline{S}_{\mu}}(\alpha^g)$, $q(\mu)^g=q(\sigma)$ and $t(\mu)=\ord(\delta)>1$.
\item If $\gamma\not=1$ (i.e., $t(\sigma)>1$) and $\ord(\gamma)\nmid g$, then there is another Lang-Steinberg map $\mu$ on $\overline{S}$ such that $\mu(\alpha^g)=\alpha^g$, $\C_{\overline{S}_{\sigma}}(\alpha)=\C_{\overline{S}_{\mu}}(\alpha^g)$, $q(\mu)^g=q(\sigma)$ and $t(\mu)=\ord(\gamma)>1$.
\item If $\gamma\not=1$ (i.e., $t(\sigma)>1$) and $\ord(\gamma)\mid g$, then writing $g=\ord(\gamma)s$ and $\alpha^s=g_2\phi^s$ with $g_2\in\overline{S}_{\mu}$ (note that necessarily $\delta=1$ here) and letting $\gamma'$ be the graph automorphism of $\overline{S}$ that agrees with $\phi^s$ on $S$, there is another Lang-Steinberg map $\mu$ on $\overline{S}$ such that $\mu(g_2\gamma')=g_2\gamma'$, $\C_{\overline{S}_{\sigma}}(\alpha)=\C_{\overline{S}_{\mu}}(g_2\gamma')$, $q(\mu)^s=q(\mu)^{g/\ord(\gamma)}=q(\sigma)$ and $t(\mu)=1$.
\end{enumerate}

Now assume additionally that $\overline{S}$ is of type $B_2$, $F_4$ or $G_2$. Let $\alpha=s\phi$ be an automorphism of $S$ with $s\in\overline{S}_{\sigma}$ and $\phi$ a graph-field automorphism of $S$. Set $g:=\ord(\phi)$. Then there is another Lang-Steinberg map $\mu$ on $\overline{S}$ such that $\mu(\alpha^g)=\alpha^g$, $\C_{\overline{S}_{\sigma}}(\alpha)=\C_{\overline{S}_{\mu}}(\alpha^g)$ and $q(\mu)^g=q(\sigma)$.\qed
\end{propposition}

In its given form, Proposition \ref{hartleyProp} allows us to determine the isomorphism type of $\overline{S}_{\mu}$ from $S=O^{p'}(\overline{S}_{\sigma})$, the field resp.~graph-field component order $g$, the defining characteristic $p$ (which, except in a few small cases, is uniquely determined by the abstract group isomorphism type of $S$) and the knowledge of which of the cases that are distinguished applies.

Let us now discuss the general methods for obtaining lower bounds on centralizer orders of inner diagonal automorphisms. As mentioned above, these are essentially due to Hartley and Kuzucuo\u{g}lu from \cite{HK91a}, but the part of their paper where the methods were introduced was concentrated on inner automorphisms only. The arguments for inner diagonal automorphisms in general are mostly analogous, but we give them here for the reader's convenience. Let $S=\leftidx{^t}X_r(p^{f\cdot t})=O^{p'}(X_r(\overline{\IF_p})_{\sigma})$ be a finite simple group of Lie type, and let $\alpha\in X_r(\overline{\IF_p})_{\sigma}$ be an inner diagonal automorphism of $S$. Set $\overline{q}(\sigma):=\min(\{q(\sigma)^e\mid e\in\IN^+\}\cap\IZ)$, so that $\overline{q}(\sigma)=q(\sigma)$ unless $S$ is one of the Suzuki or Ree groups, in which case $\overline{q}(\sigma)=q(\sigma)^2$. We make a case distinction:

\begin{enumerate}
\item If $\alpha$ is semisimple, i.e., if $p\nmid\ord(\alpha)$, then $\C_{X_r(\overline{\IF_p})_{\sigma}}(\alpha)$ contains the subgroup $T_{\sigma}$ of $\sigma$-fixed points of some $\sigma$-invariant maximal torus $T$ of $X_r(\overline{\IF_p})$. Therefore and by \cite[Lemma 3.3]{Har92a}, we then have $|\C_{X_r(\overline{\IF_p})_{\sigma}}(\alpha)|\geq(q(\sigma)-1)^r$.
\item If $\alpha$ is not semisimple, i.e., if $p\mid\ord(\alpha)$, then write $\alpha=\beta\gamma$ with $\beta\in X_r(\overline{\IF_p})_{\sigma}$ semisimple, $1\not=\gamma\in X_r(\overline{\IF_p})_{\sigma}$ unipotent and $[\beta,\gamma]=1$ (Jordan decomposition). Since $\gamma\in\C_{X_r(\overline{\IF_p})_{\sigma}}(\beta)$, we have that the group $P:=O^{p'}(\C_{X_r(\overline{\IF_p})_{\sigma}}(\beta))$ is nontrivial. Hence by \cite[Lemma 3.3 and its proof, Theorem 4.2]{HK91a}, $P$ is a central product $U_1\cdots U_r$ with $r\geq 1$ and such that for each $j\in\{1,\ldots,r\}$, there exists a simple linear algebraic group $S_j$ over $\overline{\IF_p}$, not necessarily of adjoint type, and a Lang-Steinberg map $\sigma_j$ on $S_j$ such that $U_j$ is a quotient of $O^{p'}((S_j)_{\sigma_j})$ by a subgroup of its center and $q(\sigma_j)$ is a power of $q(\sigma)$ (note that this last property is not stated in \cite[Lemma 3.3]{HK91a}, but it follows from the last paragraph of its proof, which is also the last paragraph of \cite[Section 3]{HK91a}). Consequently, $\overline{q}(\sigma_j)$ is a power of $\overline{q}(\sigma)$. Now since $\gamma\in P$, we can write $\gamma=\gamma_1\cdots\gamma_r$ with $\gamma_j$ contained in some Sylow $p$-subgroup $Q_j$ of $U_j$, $j=1,\ldots,r$. By \cite[Lemma 3.2]{HK91a} and the fact that $\zeta{P}$ is a $p'$-group, we find that $Q_j$ is isomorphic to a Sylow $p$-subgroup of $O^{p'}((\widetilde{S_j})_{\widetilde{\sigma_j}})$, where $\widetilde{S_j}$ is the simple linear algebraic group of adjoint type over $\overline{\IF_p}$ of the same isogeny type as $S_j$, and $\widetilde{\sigma_j}$ is a Lang-Steinberg map of $\widetilde{S_j}$ such that $q(\widetilde{\sigma_j})=q(\sigma_j)$ and $t(\widetilde{\sigma_j})=t(\sigma_j)$. Therefore, as in \cite[proof of Theorem A1, p.~320]{HK91a}, for each $j\in\{1,\ldots,r\}$, we have the inclusions $\zeta{Q_j}\leq\C_{Q_j}(\gamma_j)\leq\C_{X_r(\overline{\IF_p})_{\sigma}}(\alpha)$ and $\zeta{Q_j}$ contains an isomorphic copy of the additive group of $\IF_{\overline{q}(\sigma_j)}$. In particular, $|\C_{X_r(\overline{\IF_p})_{\sigma}}(\alpha)|\geq\MCS(Q_1)\cdot\ord(\beta)\geq \overline{q}(\sigma)$.
\end{enumerate}

Let us condense this information into a convenient overview of lower bounds on centralizer orders that we will use throughout the proof of Lemma \ref{mainLem4}:

\begin{propposition}\label{boundsProp}
Let $S=O^{p'}(X_r(\overline{\IF_p})_{\sigma})$ be a finite simple group of Lie type.

\begin{enumerate}
\item We have $\MCS(X_r(\overline{\IF_p})_{\sigma})\geq q(\sigma)-1$.
\item If $r>1$, then $\MCS(X_r(\overline{\IF_p})_{\sigma})\geq q(\sigma)$.
\item If $r>2$, then $\MCS(X_r(\overline{\IF_p})_{\sigma})\geq 2\cdot q(\sigma)$
\end{enumerate}
\end{propposition}

\begin{proof}
For (1): This is clear by the above considerations.

For (2): This is clear by the above considerations and a simple case distinction of \enquote{$q(\sigma)\leq 2$} (where the inequality is trivial) versus \enquote{$q(\sigma)>2$} (in which case $q(\sigma)\geq\sqrt{8}$ and thus $(q(\sigma)-1)^r\geq q(\sigma)$).

For (3): We make a case distinction:

\begin{enumerate}
\item Case: $q(\sigma)<2$. Then $q(\sigma)=\sqrt{2}$, $S=\leftidx{^2}F_4(2)=\Inndiag(S)$, and and it is sufficient to show that $\MCS(S)\geq 3$. But each involution $s\in S$ lies in some Sylow $2$-subgroup $P$ of $S$. Moreover, $\C_S(s)$ contains $\C_P(s)$, which in turn contains $\zeta P$, which contains an isomorphic copy of $\IZ/2\IZ$. But $|P|=2^{12}>2^1$, and so $\C_P(s)$ certainly is a proper supergroup of that copy of $\IZ/2\IZ$ (which one sees in a simple case distinction \enquote{$s\in\zeta P$} versus \enquote{$s\notin\zeta P$}). Hence involution centralizers in $S$ have order divisible by $4$, and all other element centralizers clearly have order at least $3$ as well.

\item Case: $q(\sigma)=2$. Then it is sufficient to show that $\MCS(X_r(\overline{\IF_2})_{\sigma})\geq 4$. But by looking through the list of finite simple groups of Lie type of untwisted Lie rank at least $3$ and the formulas for their orders, one sees that $2^23^2\mid|S|\mid|\Inndiag(S)|$, and so by an argument as in the previous case, one sees that involutions resp.~elements of order $3$ in $\Inndiag(S)$ have centralizer orders divisible by $4$ resp.~$9$, and clearly, all other elements of $\Inndiag(S)$ also have centralizer order at least $4$.

\item Case: $q(\sigma)>2$. Then $q(\sigma)\geq\sqrt{8}$, and so it is easy to see that $(q(\sigma)-1)^r\geq 2q(\sigma)$. Therefore, we only need to consider non-semisimple automorphisms $\alpha=\beta\gamma$ as above. If $\beta\not=1$, then the centralizer is of order at least $2\overline{q}(\sigma)\geq 2q(\sigma)$ by the above considerations, so assume $\beta=1$, so that $\alpha=\gamma$ is unipotent. Then the centralizer in $X_r(\overline{\IF_p})_{\sigma}$ of $\alpha$ contains the centralizer $C$ of $\alpha$ in some Sylow $p$-subgroup $Q$ of $X_r(\overline{\IF_p})_{\sigma}$ such that $\alpha\in Q$, and $C$ contains $\zeta{Q}$, which contains an isomorphic copy of $(\IF_{\overline{q}(\sigma)},+)$ by \cite[proof of Theorem A1, p.~320]{HK91a}. But since $r>1$, $Q$ must be a proper supergroup of that additive field group copy, and so $C$ certainly is a proper supergroup of it as well. So $|C|$, and thus $|\C_{X_r(\overline{\IF_p})_{\sigma}}(\alpha)|$, is a proper integer multiple of $\overline{q}(\sigma)$, and the assertion follows.
\end{enumerate}
\end{proof}

As a final preparatory remark before starting with the proof of Lemma \ref{mainLem4}, we note that since we do not expect the constant $\frac{18}{19}$ to be optimal even as an upper bound for $\h(S)$, we indicate in each of the many cases that need to be distinguished an upper bound on $\h(S)$ that our arguments for that particular case give (and which is usually smaller than $\frac{18}{19}$), so readers who would like to try and prove a better upper constant upper bound on $\h(S)$ can easily identify the cases where more work needs to be done.

\begin{proof}[Proof of Lemma \ref{mainLem4}]
We go through the different types of nonabelian finite simple groups.

If $S$ is sporadic, we claim that $\max_{\alpha\in\Aut(S)}{\alpha^{\Aut(S)}}\leq\frac{2}{5}|S|$. Indeed, by $|\Out(S)|\leq2$, it suffices to show that for each $\alpha\in\Aut(S)$, $|\C_{\Aut(S)}(\alpha)|\geq5$. This can be checked case by case using the ATLAS of Finite Group Representations \cite{ATLAS}.

Now assume that $S$ is alternating, say $S=\Alt_m$ for some $m\geq5$. Let us first verify the asymptotic statement (2,i). We claim that the function

\[
f_1:\left(0,1\right]\rightarrow\left(0,\infty\right),\rho\mapsto\max\{6,\lfloor\frac{1}{\rho}\cdot\lfloor\frac{2}{\rho}\rfloor\cdot(\lfloor\frac{2}{\rho}\rfloor+1)\rfloor+1\}
\]

has the property that for all $\rho\in\left(0,1\right]$: If $\h(\Alt_m)\geq\rho$, then $m\leq f_1(\rho)$. Indeed, assume $m>f_1(\rho)$, so that in particular $m\geq7$ and thus $\Aut(\Alt_m)=\Sym_m$. We claim that $\h(\Alt_m)<\rho$, which by the commutativity of $\Out(\Alt_m)\cong\IZ/2\IZ$ is equivalent to $\max_{\sigma\in\Sym_m}{|\sigma^{\Sym_m}|}<\rho|\Alt_m|$. Let $\sigma\in\Sym_m$. If $\ord(\sigma)>\frac{2}{\rho}$, then $|\C_{\Sym_m}(\sigma)|>\frac{2}{\rho}$, and so $|\sigma^{\Sym_m}|<\frac{\rho}{2}|\Sym_m|=\rho|\Alt_m|$. We may thus assume that $\ord(\sigma)\leq\frac{2}{\rho}$. Then every cycle of $\sigma$ has length at most $\frac{2}{\rho}$, and so $\sigma$ has at least

\[
\frac{m}{\sum_{i=1}^{\lfloor 2/\rho\rfloor}{i}}=\frac{m}{\frac{1}{2}\lfloor\frac{2}{\rho}\rfloor\cdot(\lfloor\frac{2}{\rho}\rfloor+1)}
\]

cycles of a common length, whence

\[
|\C_{\Sym_m}(\sigma)|\geq\lceil\frac{m}{\frac{1}{2}\lfloor\frac{2}{\rho}\rfloor\cdot(\lfloor\frac{2}{\rho}\rfloor+1)}\rceil!\geq\frac{m}{\frac{1}{2}\lfloor\frac{2}{\rho}\rfloor\cdot(\lfloor\frac{2}{\rho}\rfloor+1)}>\frac{2}{\rho},
\]

so that again, $|\sigma^{\Sym_m}|<\rho|\Alt_m|$, as required.

Now we verify statement (1) for $S=\Alt_m$. First, one checks that $\MCS(\Aut(\Alt_5))=\MCS(\Sym_5)=4$ and $\MCS(\Aut(\Alt_6))=6$ (for the latter, one can refer to the ATLAS of Finite Group Representations \cite{ATLAS}), which entail $\h(\Alt_5)=\frac{1}{2}$ and $\h(\Alt_6)=\frac{3}{4}$. So we may assume $m\geq7$, and we assert that then $\h(\Alt_m)\leq\frac{2}{5}$. Since $|\Out(S)|=2$, this is equivalent to showing that $\MCS(\Sym_m)\geq5$. Let $\sigma\in\Sym_m$. If $\ord(\sigma)=1$ or $\ord(\sigma)\geq5$, it is clear that $|\C_{\Sym_m}(\sigma)|\geq 5$, so assume $\ord(\sigma)\in\{2,3,4\}$. If $\ord(\sigma)=2$, then $\sigma$ has at least three cycles of a common length, so $|\C_{\Sym_m}(\sigma)|\geq 3!=6>5$. And if $\ord(\sigma)\in\{3,4\}$, then $|\C_{\Sym_m}(\sigma)|\geq\ord(\sigma)\cdot\MCS(\Sym_{m-
\ord(\sigma)})\geq\ord(\sigma)\cdot 2>5$.

Now we turn to the finite simple groups $S=\leftidx{^t}X_r(p^{f\cdot t})$ of Lie type. We first verify the asymptotic statement (2,ii). We will show the following, which is actually stronger than statement (2,ii): As either $r\to\infty$ or $p\to\infty$, the maximum conjugacy class length in $\Aut(S)$ is in $\o(|S|)$, or equivalently, $\MCS(\Aut(S))\in\omega(|\Out(S)|)$. First, assume that $r\to\infty$. Then we can restrict our attention to groups of untwisted Lie rank at least $9$, which are all classical, and so, by Proposition \ref{hartleyProp}, we then have

\begin{align*}
&\MCS(\Aut(\leftidx{^t}X_r(p^{f\cdot t})))\geq \\
&\begin{cases}\min_{g\mid f}(g\cdot\MCS(\Inndiag(X_r(p^{f/g})))), & \text{if }X_r\in\{B_s,C_s\mid s\geq9\}, \\ \min_{g\mid f,u\in\{1,2\}}(g\cdot\MCS(\Inndiag(\leftidx{^u}X_r(p^{uf/g})))), & \text{if }X_r\in\{A_s,D_s\mid s\geq9\}\text{ and }t=1, \\ \min_{g\mid 2f}(g\cdot\MCS(\Inndiag(\leftidx{^{2-\delta_{[2\mid g]}}}X_r(p^{2f/g})))), & \text{if }X_r\in\{A_s,D_s\mid s\geq9\}\text{ and }t=2.\end{cases}
\end{align*}

By \cite[Theorem 6.15]{FG12a} and using that $1+\log_{p^{f/g}}(r)\leq r$ for $p\in\IP$, $f\in\IN^+$, $g\mid f$ and $r\geq9$, we therefore have, for some universal constant $A>0$ and all $r\geq9$, $\MCS(\Aut(S))\geq\min_{g\mid f}{g\cdot\frac{p^{f/g\cdot(r-1)}}{Ar}}=f\cdot\frac{p^{r-1}}{Ar}\in\omega(r\cdot f)$ as $\max\{p,r\}\to\infty$, and since $|\Out(S)|\in\O(r\cdot f)$ as $|S|\to\infty$, we are done in this case.

Let us now assume that $p\to\infty$. By the proof for $r\to\infty$, we may also assume that $r\leq 8$, and so $|\Out(S)|\in\O(f)$ as $|S|\to\infty$. But by Proposition \ref{boundsProp}, we have $\MCS(\Aut(S))\in\Omega(p^f)$ as $|S|\to\infty$, and so $\MCS(\Aut(S))\in\omega(|\Out(S)|)$ as $p\to\infty$.

Let us now turn to the proof of statement (1) for finite simple groups $S$ of Lie type, i.e., that $\h(S)\leq\frac{18}{19}$. We go through various cases, and in each of them, we show that $\h(S)\leq c$ for some case-dependent explicit constant $c\in\left(0,\frac{18}{19}\right]$. Throughout this, we always assume that $\alpha$ is an arbitrary, but fixed automorphism of $S$ and that $g$ is the order of the field component of $\alpha$ \emph{unless} $S$ is one of $B_2(p^f)$, $G_2(p^f)$ or $F_4(p^f)$, in which case $g$ is assumed to be the order of the graph-field component of $\alpha$ (like in Proposition \ref{hartleyProp}; note that as Suzuki and Ree groups do not have nontrivial graph automorphisms, the field and graph-field component of an automorphism of any of them coincide).

\begin{enumerate}
\item Case: $S=A_1(p^f)=\PSL_2(p^f)$, so $|\Out(S)|=(2,p^f-1)f$ and $g\mid f$. We claim that then $\h(S)\leq\frac{3}{4}$. First, assume $p=2$, so that $|\Out(S)|=f$ and it suffices to show $|\C_{\Aut(S)}(\alpha)|\geq\frac{4}{3}f$. By Propositions \ref{hartleyProp}(1) and \ref{boundsProp}(1), $|\C_{\Aut(S)}(\alpha)|\geq g\cdot\MCS(\PGL_2(2^{f/g}))\geq g\cdot(2^{f/g}-1)$, and so it is sufficient to have either $\MCS(\PGL_2(2^{f/g}))\geq\frac{4}{3}\frac{f}{g}$ or the stronger $2^{f/g}\geq\frac{4}{3}\frac{f}{g}+1$. The latter holds as long as $\frac{f}{g}\geq 2$, and for $\frac{f}{g}=1$, the former is clearly true. Similarly, for $p>2$, where $|\Out(S)|=2f$, it suffices to have either $\MCS(\PGL_2(p^{f/g}))\geq\frac{8}{3}\frac{f}{g}$ or $p^{f/g}\geq\frac{8}{3}\frac{f}{g}+1$, and the latter holds unless $p=3$ and $\frac{f}{g}=1$, in which case the former holds ($\MCS(\PGL_2(3))=3$ by GAP \cite{GAP4}).
\item Case: $S=A_2(p^f)=\PSL_3(p^f)$, so $|\Out(S)|=(3,p^f-1)2f$ and $g\mid f$. We claim that then $\h(S)\leq\frac{12}{13}$. First, assume $p=3$, so that $|\Out(S)|=2f$. By Propositions \ref{hartleyProp}(1,2) and \ref{boundsProp}(2), we have

\[
|\C_{\Aut(S)}(\alpha)|\geq g\cdot\min\{\MCS(\PGL_3(3^{f/g})),\MCS(\PGU_3(3^{f/g}))\}\geq g\cdot 3^{f/g},
\]

so it is sufficient to have $3^{f/g}\geq\frac{13}{6}\frac{f}{g}$, which holds. Similarly, for $p\geq5$, using that $|\Out(S)|\leq 6f$, we see that it is sufficient to have either

\[
\min\{\MCS(\PGL_3(p^{f/g})),\MCS(\PGU_3(p^{f/g}))\}\geq\frac{13}{2}\frac{f}{g}
\]

or the stronger $p^{f/g}\geq \frac{13}{2}\frac{f}{g}$. The latter is true unless $p=5$ and $\frac{f}{g}=1$, for which one checks with GAP \cite{GAP4} that $\min\{\MCS(\PGL_3(5)),\MCS(\PGU_3(5))\}=\min\{16,21\}=16\geq\frac{13}{2}$. We are thus left with the case $p=2$, which will require some more delicate handling. First, since still $|\Out(S)|\leq6f$, we are satisfied if either $\min\{\MCS(\PGL_3(2^{f/g})),\MCS(\PGU_3(2^{f/g}))\}\geq \frac{13}{2}\frac{f}{g}$ or the stronger $2^{f/g}\geq \frac{13}{2}\frac{f}{g}$ holds. The latter is valid for $\frac{f}{g}\geq 6$, and for settling the cases $\frac{f}{g}\in\{3,4,5\}$, one computes $\MCS(\PGL_3(2^{f/g}))$ and $\MCS(\PGU_3(2^{f/g}))$ with GAP \cite{GAP4} (except for $\MCS(\PGU_3(32))$, the computation of which required too much RAM for the author's computer to handle, but \cite[Theorem 6.7]{FG12a} gives the sufficient lower bound $\MCS(\PGU_3(32))\geq\frac{32^3}{33}\sqrt{\frac{1-1/32^2}{\e(2+\log_{32}(4))}}=388.5715\ldots$; this approach also works for $\MCS(\PGU_3(16))$, where the GAP computations on the author's computer did finish, but took quite long), and under the assumption $2\nmid f$, where $|\Out(S)|=2f$ only, one can also settle the cases $\frac{f}{g}\in\{1,2\}$ with GAP by verifying that $\MCS(\PGL_3(4))=12$, $\MCS(\PGU_3(4))=13$, $\MCS(\PGL_3(2))=3$ and $\MCS(\PGU_3(4))=4$. We may thus henceforth assume that $2\mid f$ and $\frac{f}{g}\in\{1,2\}$. By \cite[Theorem 2.5.12(c--e), p.~58]{GLS98a}, each of the subsets $\Outdiag(S),\Phi_S,\Gamma_S\subseteq\Out(S)$ is actually a cyclic subgroup of $\Out(S)$ here, of order $3$, $f$ and $2$ respectively. Let $d$, $\phi$ and $\tau$ be fixed generators of these groups in the given order. Then, identifying $d$ with a lift of it in $\Aut(S)$, the automorphisms $d^a\phi^b\tau^c$, where $a\in\{0,1,2\}$, $b\in\{0,1,\ldots,f-1\}$ and $c\in\{0,1\}$, form a set of representatives for the $S$-cosets in $\Aut(S)$. Assume that $a,b,c$ are such that $\alpha$ is in the corresponding $S$-coset. We distinguish two subcases:

\begin{enumerate}
\item Subcase: $a=0$, and either $b$ is even and $c=0$, or $b$ is odd and $c=1$. By \cite[Theorem 2.5.12(b--e,g,i), p.~58]{GLS98a}, this is just the case where $\alpha$ is out-central, i.e., where $\alpha^{\Aut(S)}$ is contained in a single $S$-coset. But $\phi$ and $\tau$ actually commute not only in $\Out(S)$, but in $\Aut(S)$, and so the coset representative $\beta:=\phi^b\tau^c\in S\alpha$ has order $g=\ord(\phi^b)$. By Proposition \ref{hartleyProp}(1,2), we therefore get that $|\C_{\Aut(S)}(\beta)|\geq 2f\cdot|\C_{\PGL_3(2^f)}(\alpha)|=2f\cdot|\C_{A_2(\overline{\IF_2})_{\mu}}(1)|=2f\cdot|A_2(\overline{\IF_2})_{\mu}|>6f$, so $\beta^{\Aut(S)}$ is a proper subset of $S\beta=S\alpha$, and therefore, $\alpha^{\Aut(S)}$ must also be a proper subset of it, whence $|\C_{\Aut(S)}(\alpha)|>6f$ necessarily. Noting that $|\C_{\Aut(S)}(\alpha)|$ is also an integer multiple of $g\cdot|\C_{\PGL_2(2^f)}(\alpha)|$, we can conclude this subcase in the following subsubcase distinction:

\begin{itemize}
\item Subsubcase: $|\C_{\Aut(S)}(\alpha)|\geq 3g\cdot|\C_{\PGL_2(2^f)}(\alpha)|$. Then for $\frac{f}{g}=2$,

\[
|\C_{\Aut(S)}(\alpha)|\geq \frac{3}{2}f\cdot\min\{\MCS(\PGL_3(4)),\MCS(\PGU_3(4))\}=18f,
\]

and for $\frac{f}{g}=1$,

\[
|\C_{\Aut(S)}(\alpha)|\geq 3f\cdot\min\{\MCS(\PGL_3(2)),\MCS(\PGU_3(2))\}=9f.
\]
\item Subsubcase: $|\C_{\Aut(S)}(\alpha)|=2g\cdot|\C_{\PGL_2(2^f)}(\alpha)|$. Then for $\frac{f}{g}=2$, we get that

\[
|\C_{\Aut(S)}(\alpha)|\geq f\cdot\min\{\MCS(\PGL_3(4)),\MCS(\PGU_3(4))\}=12f,
\]

and for $\frac{f}{g}=1$, we get that $|\C_{\Aut(S)}(\alpha)|\geq 2f\cdot m$, where $m>3$ is the order of some element centralizer in one of the groups $\PGL_3(2)$ or $\PGU_3(2)$, so in particular, $m$ is an integer and therefore $m\geq4$, which yields $|\C_{\Aut(S)}(\alpha)|\geq 8f$.
\item Subsubcase: $|\C_{\Aut(S)}(\alpha)|=g\cdot|\C_{\PGL_2(2^f)}(\alpha)|$. Then for $\frac{f}{g}=2$, we get that $|\C_{\Aut(S)}(\alpha)|\geq \frac{f}{2}\cdot m$ for some integer $m>12$, so that $|\C_{\Aut(S)}(\alpha)|\geq\frac{13}{2}f$, and for $\frac{f}{g}=1$, we have $|\C_{\Aut(S)}(\alpha)|\geq f\cdot m'$ for some integer $m'>6$, so $|\C_{\Aut(S)}(\alpha)|\geq 7f$.
\end{itemize}

\item Subcase: $\alpha$ is \emph{not} out-central, i.e., $a\not=0$, or $2\mid b$ and $c=1$, or $2\nmid b$ and $c=0$. Then $\alpha^{\Aut(S)}$ distributes evenly among at least two $S$-cosets. Hence for $\frac{f}{g}=2$, where we have $|\alpha^{\Aut(S)}|\leq|S|$ by

\[
\min\{\MCS(\PGL_3(4)),\MCS(\PGU_3(4))\}=12=2\cdot 6,
\]

we are done since then the proportion of elements of $\alpha^{\Aut(S)}$ per coset is at most $\frac{1}{2}\leq\frac{12}{13}$. So assume $\frac{f}{g}=1$. Then $|\alpha^{\Aut(S)}|\leq 2|S|$, and so we are done if $|\alpha^{\Out(S)}|\geq 3$, whence we may assume $|\alpha^{\Out(S)}|=2$. We make a subsubcase distinction:

\begin{itemize}
\item Subsubcase: $a=0$. Then as above, the representative $\beta:=\phi^b\tau^c$ of the $S$-coset of $\alpha$ has order exactly $g=f$, so its centralizer is certainly of order larger than $3f$, whence $\beta^{\Aut(S)}$, and thus $\alpha^{\Aut(S)}$, cannot be a union of two $S$-cosets, so $|\C_{\Aut(S)}(\alpha)|>3f$. Hence if $|\C_{\Aut(S)}(\alpha)|\geq 2f|\C_{\PGL_3(2^f)}(\alpha)|$, then $|\C_{\Aut(S)}(\alpha)|\geq 6f$, so the density of $\alpha^{\Aut(S)}$ per $S$-coset is at most $\frac{1}{2}$, and if $|\C_{\Aut(S)}(\alpha)|=f\cdot|\C_{\PGL_3(2^f)}(\alpha)|$, then $|\C_{\Aut(S)}(\alpha)|=f\cdot m$ for some integer $m>3$, so $|\C_{\Aut(S)}(\alpha)|\geq 4f$ and $\alpha^{\Aut(S)}$ has density per $S$-coset at most $\frac{3}{4}$.
\item Subsubcase: $a\not=0$. If $c=0$, then by \cite[Theorem 2.5.12(b--e,g,i), p.~58]{GLS98a}, we have that the image of $\alpha$ in $\Out(S)/\langle\phi^2\rangle\cong\D_{12}$ is of order $2$, and so $|\alpha^{\Out(S)/\langle\phi^2\rangle}|=3$, whence $|\alpha^{\Out(S)}|\geq3$, contradicting our assumption that it is $2$. Therefore, $c=1$ and by Proposition \ref{hartleyProp}(2), we conclude that $|\C_{\Aut(S)}(\alpha)|\geq f\cdot\MCS(\PGU_3(2))=4f$, so the density of $\alpha^{\Aut(S)}$ per $S$-coset is at most $\frac{3}{4}$.
\end{itemize}
\end{enumerate}

\item Case: $S=B_2(p^f)=\POmega_5(p^f)$, so $|\Out(S)|=2f$ and $g\mid(1+\delta_{[p=2]})f$. We claim that then $\h(S)\leq\frac{3}{5}$. First, assume that $g\mid f$. By the final paragraph in Proposition \ref{hartleyProp}, we have $|\C_{\Aut(S)}(\alpha)|\geq g\cdot|\C_{\overline{G}_{\mu}}(\alpha^g)|$, and since $q(\mu)=p^{f/g}\in\IZ$, the fixed point subgroup $\overline{G}_{\mu}$ must be isomorphic to the inner diagonal automorphism group $\PGO_5(p^{f/g})$ of $\POmega_5(p^{f/g})$. It follows that $|\C_{\Aut(S)}(\alpha)|\geq g\cdot\MCS(\PGO_5(p^{f/g}))$, and so we are satisfied if either $\MCS(\PGO_5(p^{f/g}))\geq \frac{10}{3}\frac{f}{g}$ or the (by Proposition \ref{boundsProp}(2)) stronger inequality $p^{f/g}\geq \frac{10}{3}\frac{f}{g}$ holds. The latter is true for $p\geq 5$, for $p=3$ and $\frac{f}{g}\geq 2$, and for $p=2$ and $\frac{f}{g}\geq 4$, and the former can be verified to hold in the other cases with GAP \cite{GAP4} (by computing $\MCS(\GO_5(p^{f/g})/\zeta\GO_5(p^{f/g}))$; for $p=2$, one can also exploit the isomorphisms $B_2(2^t)\cong C_2(2^t)\cong\Sp_4(2^t)$). Now assume that $g\nmid f$. Then $p=2$, and $g=2g'$, where $g'\mid f$ and $\frac{f}{g'}$ is odd. By the final pararaph in Proposition \ref{hartleyProp}, we have $|\C_{\Aut(S)}(\alpha)|\geq g\cdot|\C_{\overline{G}_{\mu}}(\alpha^g)|$, and since $q(\mu)=p^{f/g}\notin\IZ$, the fixed point subgroup $\overline{G}_{\mu}$ must be isomorphic to (the inner diagonal automorphism group of) $\leftidx{^2}B_2(2^{f/g'})$. We are thus satisfied if either $\MCS(\leftidx{^2}B_2(2^{f/g'}))\geq\frac{5}{3}\frac{f}{g'}$ or the (by Proposition \ref{boundsProp}(2)) stronger inequality $2^{f/(2g')}\geq\frac{5}{3}\frac{f}{g'}$ holds. The latter is true for $\frac{f}{g'}\geq 7$, and the validity of the former for $\frac{f}{g'}\in\{1,3,5\}$ can be checked with GAP \cite{GAP4}.

\item Case: $S=G_2(p^f)$, so $|\Out(S)|=(1+\delta_{[p=3]})f$ and $g\mid|\Out(S)|$. We claim that then $\h(S)\leq\frac{2}{3}$. First, assume that $g\mid f$. Then if $p\not=3$, we have $|\Out(S)|=f$, and by the final paragraph in Proposition \ref{hartleyProp} and by Proposition \ref{boundsProp}(2), we get that $|\C_{\Aut(S)}(\alpha)|\geq g\cdot\MCS(G_2(p^{f/g}))\geq g\cdot p^{f/g}\geq 2f\geq\frac{3}{2}f$. If, on the other hand, $p=3$, then $|\Out(S)|=2f$, and by the same argument, we get that $|\C_{\Aut(S)}(\alpha)|\geq g\cdot 3^{f/g}\geq 3f$. Now assume $g\nmid f$. Then $p=3$, and $g=2g'$, where $g'\mid f$ and $\frac{f}{g'}$ is odd. By the final paragraph in Proposition \ref{hartleyProp} and by Proposition \ref{boundsProp}(2), we get that $|\C_{\Aut(S)}(\alpha)|\geq g\cdot\MCS(\leftidx{^2}G_2(3^{f/g'}))\geq g\cdot 3^{f/(2g')}$, which is greater than or equal to $3f$, since $3^{f/(2g')}\geq\frac{3}{2}\frac{f}{g'}$.

\item Case: $S=\leftidx{^2}A_2(p^{2f})=\PSU_3(p^f)$, so $|\Out(S)|=(3,p^f+1)\cdot 2f$ and $g\mid 2f$. We claim that then $\h(S)\leq\frac{6}{7}$. By Proposition \ref{hartleyProp}(3,4), we have the following:

\begin{itemize}
\item If $2\mid g$, then $|\C_{\Aut(S)}(\alpha)|\geq g\cdot\MCS(\PGL_3(p^{2f/g}))$.
\item If $2\nmid g$, then $|\C_{\Aut(S)}(\alpha)|\geq g\cdot\MCS(\PGU_3(p^{f/g}))$.
\end{itemize}

Let us first assume $p\geq 3$. By Proposition \ref{boundsProp}(2), we are done if $p^{f/g}\geq 7\frac{f}{g}$, which is true for $p\geq 7$, and for $p=5$ and $\frac{f}{g}\geq 2$. For $p=5$ and $\frac{f}{g}=1$, note that if $2\mid g=f$, then by the above, we are actually satisfied if $5^{2f/g}=25$ is at least $7\cdot 1=7$, which is true; if, on the other hand, $2\nmid g=f$, we are satisfied if $\MCS(\PGU_3(5))\geq 7$, which can be checked with GAP \cite{GAP4}. Finally, for $p=3$, we have $|\Out(S)|=2f$ only, so that we are satisfied with $3^{f/g}\geq\frac{7}{3}\frac{f}{g}$, which is true.

Let us now assume $p=2$. If $2\mid f$, then $|\Out(S)|=2f$ only, and so we are satisfied if for $t\in\{1,2\}$ and all $e\mid f$, $\MCS(\Inndiag(\leftidx{^t}A_2(2^{2f/(te)})))\geq\frac{7}{3}\frac{f}{e}$, which holds by Proposition \ref{boundsProp}(2) and the computed $\MCS$-values for $\PGU_3(2^d)$, $d\in\{1,2\}$, from Case 2 (\enquote{$S=A_2(p^f)$}). So we may henceforth assume $2\nmid f$, in which case $|\Out(S)|=6f$. We make a subcase distinction:

\begin{enumerate}
\item Subcase: $2\mid g$. Then since $2\nmid f$, the field component of $\alpha$ is an odd exponent power of any fixed generator $\phi$ of $\Phi_S$, and so by \cite[Theorem 2.5.12(c,g), p.~58]{GLS98a}, $\alpha$ is not out-central. Therefore, we are actually satisfied with $|\C_{\Aut(S)}(\alpha)|\geq\frac{7}{2}f$, which follows from $\MCS(\PGL_3(2^{2f/g}))\geq\frac{7}{2}\frac{f}{g}$, which holds true by Proposition \ref{boundsProp}(2).
\item Subcase: $2\nmid g$. Then $g\mid f$. We make a subsubcase distinction:

\begin{itemize}
\item Subsubcase: The outer diagonal component of the image of $\alpha$ in $\Out(S)$ is nontrivial. Then $\alpha$ is again not out-central, so we are satisfied with $|\C_{\Aut(S)}(\alpha)|\geq\frac{7}{2}f$, which is implied by $\MCS(\PGU_3(2^{f/g}))\geq\frac{7}{2}\frac{f}{g}$, and this holds by Proposition \ref{boundsProp} and by computing the values of $\MCS(\PGU_3(2^d))$, $d\in\{1,2,3\}$, with GAP \cite{GAP4}.
\item Subsubcase: The outer diagonal component of the image of $\alpha$ in $\Out(S)$ is trivial. Then $\alpha$ is out-central, and by Proposition \ref{hartleyProp}(3), the unique element $\beta$ of ${S\alpha}\cap\Phi_S$ satisfies $|\C_{\Aut(S)}(\beta)|>6f$, so that $|\C_{\Aut(S)}(\alpha)|>6f$ as well. Using that $|\C_{\Aut(S)}(\alpha)|$ is an integer multiple of $g\cdot|\PGU_3(2^f)(\alpha)|$, we can conclude with the following subsubsubcase distinction:

\begin{itemize}
\item Subsubsubcase: $|\C_{\Aut(S)}(\alpha)|\geq 2g|\PGU_3(2^f)(\alpha)|$. Then we are done if $\MCS(\PGU_3(2^{f/g}))\geq\frac{7}{2}\frac{f}{g}$, which is true (see the previous subsubcase).
\item Subsubsubcase: $|\C_{\Aut(S)}(\alpha)|=g|\PGU_3(2^f)(\alpha)|$. Then we are done if $\MCS(\PGU_3(2^{f/g}))\geq 7\frac{f}{g}$, which is true for $\frac{f}{g}\geq 2$ by Proposition \ref{boundsProp}(2), by computing $\MCS(\PGU_3(2^d))$, $d\in\{1,2,3\}$, with GAP \cite{GAP4}, and by observing that by \cite[Theorem 6.7]{FG12a}, $\MCS(\PGU_3(16))\geq93$ and $\MCS(\PGU_3(32))\geq389$. So assume $\frac{f}{g}=1$. Then we have $|\C_{\Aut(S)}(\alpha)|\geq f\cdot m$ for some integer $m>6$, and so $|\C_{\Aut(S)}(\alpha)|\geq7f$, as required.
\end{itemize}
\end{itemize}
\end{enumerate}

\item Case: $S=\leftidx{^2}B_2(2^{2f'+1})$, so $|\Out(S)|=2f'+1=2f$ and $g\mid 2f'+1$. We claim that then $\h(S)\leq\frac{3}{5}$. By the last paragraph in Proposition \ref{hartleyProp}, we know that $\C_{\leftidx{^2}B_2(2^{2f'+1})}(\alpha)=\C_{B_2(\overline{\IF_2})_{\mu}}(\alpha^g)$, where $\mu$ is a Lang-Steinberg map on $B_2(\overline{\IF_2})$ such that $q(\mu)=2^{f/g}\notin\IZ$. Therefore, $|\C_{\Aut(S)}(\alpha)|\geq g\cdot\MCS(\leftidx{^2}B_2(2^{(2f'+1)/g}))$, and so it is sufficient to have $\MCS(\leftidx{^2}B_2(2^{(2f'+1)/g}))\geq\frac{5}{3}\frac{2f'+1}{g}$, which certainly holds for $\frac{2f'+1}{g}=1$, and for $\frac{g}{2f'+1}\in\{3,5\}$, one checks with GAP \cite{GAP4} that it holds. Finally, for $\frac{2f'+1}{g}\geq 7$, it holds since by Proposition \ref{boundsProp}(2), we have $\MCS(\leftidx{^2}B_2(2^{(2f'+1)/g}))\geq \lceil2^{(2f'+1)/(2g)}\rceil\geq\frac{5}{3}\frac{2f'+1}{g}$.

\item Case: $S=\leftidx{^2}G_2(3^{2f'+1})$, so $|\Out(S)|=2f'+1=2f$ and $g\mid 2f'+1$. We claim that then $\h(S)\leq\frac{2}{5}$. As in the previous case, it is sufficient to have the bound $\MCS(\leftidx{^2}G_2(3^{(2f'+1)/g}))\geq\frac{5}{2}\frac{2f'+1}{g}$. This can be verified for $\frac{2f'+1}{g}=1$ with GAP \cite{GAP4}, and for $\frac{2f'+1}{g}=3$, it can be checked with the ATLAS of Finite Group Representations \cite{ATLAS}. Finally, for $\frac{2f'+1}{g}\geq5$, it holds by Proposition \ref{boundsProp}(2).

\item Case: $S=A_3(p^f)=\PSL_4(p^f)$, so $|\Out(S)|=(4,p^f-1)\cdot 2f$ and $g\mid f$. We claim that then $\h(S)\leq\frac{1}{2}$. First, assume $p=2$, where $|\Out(S)|=2f$ only. Using Proposition \ref{hartleyProp}(1,2), we are done if

\[
\min\{\MCS(\PGL_4(2^{f/g})),\MCS(\PGU_4(2^{f/g}))\}\geq4\frac{f}{g},
\]

which holds by Proposition \ref{boundsProp}(3) and by computing $\MCS(\PGL_4(2))=6$ and $\MCS(\PGU_4(2))=5$ with GAP \cite{GAP4}. Now assume $p\geq 3$. As $|\Out(S)|\leq 8f$, we are satisfied if $\min\{\MCS(\PGL_4(p^{f/g})),\MCS(\PGU_4(p^{f/g}))\}\geq16\frac{f}{g}$, and so by Proposition \ref{boundsProp}(3) in particular if $p^{f/g}\geq 8\frac{f}{g}$, which holds for $p\geq 11$, for $p\in\{5,7\}$ and $\frac{f}{g}\geq 2$, and for $p=3$ and $\frac{f}{g}\geq 3$. Each of the remaining cases can be checked either with GAP \cite{GAP4} or using \cite[Theorems 6.4 and 6.7]{FG12a}.

\item Case: $S=B_r(p^f)=\POmega_{2r+1}(p^f)$, $r\geq 3$, so $|\Out(S)|=(2,p^f-1)\cdot f$ and $g\mid f$. We claim that then $\h(S)\leq\frac{2}{5}$. First, assume $p=2$, where $|\Out(S)|=f$ only. By Propositions \ref{hartleyProp}(1) and \ref{boundsProp}(3), we are done if $2\cdot 2^{f/g}\geq\frac{5}{2}\frac{f}{g}$, which is true. Now assume $p\geq 3$. Since $|\Out(S)|=2f$, we are done by Propositions \ref{hartleyProp}(1) and \ref{boundsProp}(3) if $2\cdot p^{f/g}\geq 5\frac{f}{g}$, which also holds.

\item Case: $S=C_r(p^f)=\PSp_{2r}(p^f)$, $r\geq 3$, so $|\Out(S)|=(2,p^f-1)\cdot f$ and $g\mid f$. We claim that then $\h(S)\leq\frac{2}{5}$. Since $C_r(2^f)\cong B_r(2^f)$, we may assume that $p\geq 3$, which can be treated analogously to the subcase \enquote{$p\geq3$} of Case 9.

\item Case: $S=\leftidx{^2}A_3(p^{2f})=\PSU_4(p^f)$, so $|\Out(S)|=(4,p^f+1)\cdot 2f$ and $g\mid 2f$. We claim that then $\h(S)\leq\frac{1}{2}$. First, assume $p=2$. Then $|\Out(S)|=2f$ only, and so we are done if we can show that $|\C_{\Aut(S)}(\alpha)|\geq 4f$. Now if $2\nmid g$, then this is implied by $\MCS(\PGU_4(2^{f/g}))\geq 4\frac{f}{g}$, which we already checked to be true in Case 8. And if $2\mid g$, then it is implied by $\MCS(\PGL_4(2^{2f/g}))\geq 4\frac{f}{g}=2\frac{2f}{g}$, which also follows by the argument in Case 8. Now assume $p\geq 3$. As $|\Out(S)|\leq 8f$, we are satisfied if we can show that $|\C_{\Aut(S)}(\alpha)|\geq 16f$, and this can be reduced to Case 8 in a similar fashion as for $p=2$.

\item Case: $S=A_r(p^f)=\PSL_{r+1}(p^f)$, $r\geq 4$, so $|\Out(S)|=(r+1,p^f-1)\cdot 2f$ and $g\mid f$. We claim that then $\h(S)\leq\frac{10}{11}$, which we prove by showing the stronger assertion that $\min\{\MCS(\PGL_{r+1}(p^{f/g})),\MCS(\PGU_{r+1}(p^{f/g}))\}\geq\frac{11}{5}(r+1)\frac{f}{g}$, in the following two steps:

\begin{itemize}
\item First, one checks that it holds for $r\in\{4,\ldots,9\}$ using a combination of applications of Propositions \ref{hartleyProp}(1,2) and \cite[Theorems 6.4 and 6.7]{FG12a}, which cover all cases except $p=2$ and $\frac{f}{g}=1$, and GAP \cite{GAP4} computations of the exact values of $\MCS(\PGL_{r+1}(2))$ and $\MCS(\PGU_{r+1}(2))$ (the latter $\MCS$-value only needs to be computed for $r\in\{4,5,6\}$, since for $r\geq 7$, Fulman-Guralnick's bounds cover it as well) to deal with that remaining case too.
\item For $r\geq 10$, Fulman-Guralnick's bounds become strong enough to cover all cases, even $p=2$ and $\frac{f}{g}=1$, and one can conclude as follows: Let us first show that $\MCS(\PGL_{r+1}(p^{f/g}))\geq\frac{11}{5}(r+1)\frac{f}{g}$. By \cite[Theorem 6.4]{FG12a}, this $\MCS$-value is bounded from below by

\[
\frac{p^{(r+1)f/g}}{p^{f/g}-1}\cdot\frac{1-1/p^{f/g}}{\e(1+\log_{p^{f/g}}(r+2))}\geq p^{rf/g}\cdot\frac{1}{2\e(1+\log_2(r+2))},
\]

and this is greater than or equal to $\frac{11}{5}(r+1)\frac{f}{g}$ if and only if

\begin{equation}\label{eq3}
p^{r\cdot f/g}\geq\frac{22}{5}(r+1)\e(1+\log_2(r+2))\cdot\frac{f}{g}.
\end{equation}

By an inductive argument, for each prime $p$, Formula (\ref{eq3}) holds for all values of $\frac{f}{g}$ if and only if it holds for $\frac{f}{g}=1$. Moreover, Formula (\ref{eq3}) with $\frac{f}{g}=1$ holds for all primes $p$ if and only if it holds for $p=2$. We are thus left with verifying that for $r\geq 10$,

\begin{equation}\label{eq4}
2^r\geq\frac{22}{5}(r+1)\e(1+\log_2(r+2)),
\end{equation}

which can be shown by induction on $r$: It is readily verified for $r=10$, and upon replacing $r$ by $r+1$, the left-hand side in Formula (\ref{eq4}) grows by a factor of $2$, whereas the right-hand side only grows by a factor of

\begin{align*}
\frac{r+2}{r+1}\cdot\frac{1+\log_2(r+3)}{1+\log_2(r+2)} &\leq\frac{12}{11}\cdot(1+\frac{1}{1+\log_2(r+2)}) \\ &\leq\frac{12}{11}\cdot(1+\frac{1}{1+\log_2(12)})=1.3288\ldots<2.
\end{align*}

The argument for showing that $\MCS(\PGU_{r+1}(p^{f/g}))\geq\frac{11}{5}(r+1)f$ is analogous.
\end{itemize}

\item Case: $S=D_4(p^f)=\POmega_8^+(p^f)$, so $|\Out(S)|=(2,p-1)^2\cdot 6f$ and $g\mid f$. We claim that then $\h(S)\leq\frac{12}{13}$. Let us first assume $p=2$, so that $|\Out(S)|=6f$ and we are done if we can show that $|\C_{\Aut(S)}(\alpha)|\geq\frac{13}{2}f$. By \cite[Theorem 2.5.12(b,c,e), p.~58]{GLS98a}, $\Out(S)=\langle\phi\rangle\times\Gamma_S\cong\IZ/f\IZ\times\Sym_3$, where $\phi$ is a generator of $\Phi_S$. We can thus write $\alpha=s\phi^b\tau$ with $s\in S$, $b\in\{0,\ldots,f-1\}$ and $\tau\in\Sym_3$ a graph automorphism of $S$. We make a subcase distinction:

\begin{enumerate}
\item Subcase: $\tau=1$. Then by Proposition \ref{hartleyProp}(1),

\[
|\C_{\Aut(S)}(\alpha)|\geq g\cdot\MCS(\POmega_8^+(2^{f/g})),
\]

and so we are done if we can show that $\MCS(\POmega_8^+(2^{f/g}))\geq\frac{13}{2}\frac{f}{g}$. But by \cite[Theorem 6.13(2)]{FG12a}, and using that $\GO_8^+(2^{f/g})=\PGO_8^+(2^{f/g})$ is an index $2$ extension of $S=\POmega_8^+(p^f)$, we get that

\begin{align*}
\MCS(\POmega_8^+(2^{f/g})) &\geq 2^{3f/g}\cdot\sqrt{\frac{1-1/2^{f/g}}{2\e(4+\log_{2^{f/g}}(16))}} \\
&\geq 8^{f/g}\cdot\sqrt{\frac{1}{4\e(4+\log_2(16))}}=8^{f/g}\cdot\sqrt{\frac{1}{32\e}},
\end{align*}

and this is greater than or equal to $\frac{13}{2}\frac{f}{g}$ if and only if

\[
8^{f/g}\geq\frac{13}{2}\sqrt{32\e}\frac{f}{g}=60.6227\ldots\frac{f}{g},
\]

which holds for $\frac{f}{g}\geq 3$. For $\frac{f}{g}=1$, one checks with GAP \cite{GAP4} that $\MCS(\POmega_8^+(2))=7$. Finally, for $\frac{f}{g}=2$, we argue as follows: $\alpha$ is out-central, but its $S$-coset contains an element $\beta$ (e.g., $\beta=\phi^b$) such that $|\C_{\Aut(S)}(\beta)|>6f$, and so $|\C_{\Aut(S)}(\alpha)|>6f$ as well. Since $|\C_{\Aut(S)}(\alpha)|$ is also an integer multiple of $g\cdot|\C_{\POmega_8^+(2^f)}(\alpha)|=\frac{f}{2}|\C_{\POmega_8^+(2^f)}(\alpha)|$, we conclude this first subcase with the following subsubcase distinction:

\begin{enumerate}
\item $|\C_{\Aut(S)}(\alpha)|\geq f|\C_{\POmega_8^+(2^f)}(\alpha)|$. Then $\MCS(\POmega_8^+(4))\geq\frac{13}{2}$, which follows from \cite[Theorem 6.13(2)]{FG12a}, is sufficient.
\item $|\C_{\Aut(S)}(\alpha)|=\frac{f}{2}|\C_{\POmega_8^+(2^f)}(\alpha)|$. Then $|\C_{\POmega_8^+(2^f)}(\alpha)|>12$, and so indeed, $|\C_{\Aut(S)}(\alpha)|\geq\frac{13}{2}f$, as required.
\end{enumerate}

\item Subcase: $\ord(\tau)=2$. Then $|\alpha^{\Out(S)}|=3$, so $|\C_{\Aut(S)}(\alpha)|\geq\frac{13}{6}f$ is sufficient here. By Proposition \ref{hartleyProp}(1), if $2\nmid g$, then this is implied by $\MCS(\POmega_8^+(2^{f/g}))\geq\frac{13}{6}\frac{f}{g}$, and this is clear by the computations in the previous subcase. So we may assume $2\mid g$ now, for which we are done by Proposition \ref{hartleyProp}(2) if $\MCS(\POmega_8^-(2^{f/g}))\geq\frac{13}{6}\frac{f}{g}$, which can also be checked by \cite[Theorem 6.13(2)]{FG12a} and computing (with GAP \cite{GAP4}) $\MCS(\POmega_8^-(2))=9$.

\item Subcase: $\ord(\tau)=3$. Then $|\alpha^{\Out(S)}|=2$, so $|\C_{\Aut(S)}(\alpha)|\geq\frac{13}{4}f$ is sufficient here. If $3\nmid g$, then by Proposition \ref{hartleyProp}(1), this is implied by $\MCS(\POmega_8^+(2^{f/g}))\geq\frac{13}{4}\frac{f}{g}$, which is clear by the argument in the first subcase. So we may assume $3\mid g$ now, for which we are done by Proposition \ref{hartleyProp}(2) if $\MCS(\leftidx{^3}D_4(2^{3f/g}))\geq\frac{13}{4}\frac{f}{g}$. For $\frac{f}{g}\geq 2$, this holds by Proposition \ref{boundsProp}(3), and so it remains to show that $\MCS(\leftidx{^3}D_4(8))\geq 4$. But the only elements of $\leftidx{^3}D_4(8)$ whose centralizer could possibly be of order smaller than $4$ are the involutions and the elements of order $3$, and the former lie in a Sylow $2$-subgroup of $\leftidx{^3}D_4(8)$, of order $2^{12}>2$, so their centralizer order is divisible by $2^2=4$, whereas the latter lie in a Sylow $3$-subgroup of $\leftidx{^3}D_4(8)$, of order $3^4>3$, so their centralizer order is divisible by $3^2=9$.
\end{enumerate}

This concludes the argument for $p=2$, so we may assume $p\geq 3$ henceforth, so that $|\Out(S)|=24f$, and more specifically, $\Out(S)=(\Outdiag(S)\times\Phi_S)\rtimes\Gamma_S\cong((\IZ/2\IZ)^2\times\IZ/f\IZ)\rtimes\Sym_3$, and the conjugation action of $\Gamma_S\cong\Sym_3$ on $\Phi_S\cong\IZ/f\IZ$ is trivial, whereas it is faithful on $\Outdiag(S)\cong(\IZ/2\IZ)^2$. In particular, $\Phi_S\unlhd\Out(S)$ and $\Out(S)/\Phi_S\cong\Sym_4$. We make a subcase distinction:

\begin{enumerate}
\item Subcase: The graph component of $\alpha$ has order at most $2$. Then by Proposition \ref{hartleyProp}(1,2), we are done if we can show that $\MCS(\PGO_8^{\pm}(p^{f/g}))\geq 26\frac{f}{g}$. Again, we use \cite[Theorem 6.13(2)]{FG12a} (and that $|\zeta\GO_8^{\pm}(p^{f/g})|=2$) to get that $\MCS(\PGO_8^{\pm}(p^{f/g}))\geq p^{3f/g}\cdot\sqrt{\frac{1-1/p^{f/g}}{2\e(4+\log_{p^{f/g}}(16)}}\geq p^{3f/g}\cdot\sqrt{\frac{1}{3\e(4+\log_3(16))}}$, which is at least $26\frac{f}{g}$ if and only if $p^{3f/g}\geq 26\sqrt{3\e(4+\log_3(16))}\frac{f}{g}=189.6395\ldots\frac{f}{g}$. This holds true unless $p\in\{3,5\}$ and $\frac{f}{g}=1$. Moreover, one can check with GAP \cite{GAP4} that $\MCS(\PGO_8^+(3))=\MCS(\PGO_8^-(3))=40\geq 26$; for this, the author defined $\PGO_8^{\pm}(3)$ in GAP as the quotient of $\GO_8^{\pm}(3)$ by its centre and computed the character table of this group, which settles $p=3$. For $p=5$, it seems the analogous computational approach to $p=3$ was too much for the author's computer to handle, but one can work around this as follows: Let $s\in\PGO_8^{\pm}(5)$. If $s$ is semisimple, then by the arguments leading to Proposition \ref{boundsProp}, $|\C_{\PGO_8^{\pm}(5)}(s)|\geq (5-1)^4=256\geq 26$, and if $5\mid\ord(s)$, then we can argue that $|\C_{\PGO_8^{\pm}(5)}(s)|\geq 30\geq 26$, as follows: We have that $5\mid|\C_{\PGO_8^{\pm}(5)}(s)|$, and Fulman-Guralnick's bounds \cite[Theorem 6.13(2)]{FG12a} give $|\C_{\PGO_8^{\pm}(5)}(s)|\geq 21$. It therefore suffices to argue that $|\C_{\PGO_8^{\pm}(5)}(s)|\not=25$. Assume otherwise. Then $s$ lies in some Sylow $5$-subgroup of $\PGO_8^{\pm}(5)$, which is isomorphic to a Sylow $5$-subgroup of $\POmega_8^{\pm}(5)$. But with GAP \cite{GAP4}, one can check that the minimum centralizer size in such a Sylow $5$-subgroup is $625$ (for this, the author computed the character table of the Sylow $5$-subgroup after transforming it into a pc group). Hence $|\C_{\PGO_8^{\pm}(5)}(s)|\geq625$, a contradiction.
\item Subcase: The graph component of $\alpha$ has order $3$. Then the image of $\alpha$ in $\Out(S)/\Phi_S\cong\Sym_4$ has order $3$, and so $|\alpha^{\Out(S)}|\geq8$. We are therefore satisfied if we can show that $|\C_{\Aut(S)}(\alpha)|\geq\frac{13}{4}f$. If $3\nmid g$, then by Proposition \ref{hartleyProp}(1), it suffices to have $\MCS(\PGO_8^+(p^{f/g}))\geq\frac{13}{4}\frac{f}{g}$, which holds true by the arguments in the previous subcase. So assume $3\mid g$. Then by Proposition \ref{hartleyProp}(2), it suffices to have $\MCS(\leftidx{^3}D_4(p^{3f/g}))\geq\frac{13}{4}\frac{f}{g}$, which holds true by Proposition \ref{boundsProp}(3).
\end{enumerate}

\item Case: $S=F_4(p^f)$, so $|\Out(S)|=(2,p)f$ and $g\mid|\Out(S)|$. We claim that then $\h(S)\leq\frac{1}{2}$. First, assume that $g\mid f$. The claim is clear if we can show that $|\C_{\Aut(S)}(\alpha)|\geq 4f$, which by the final paragraph in Proposition \ref{hartleyProp} holds if $\MCS(F_4(p^{f/g}))\geq4\frac{f}{g}$, which is clear by Proposition \ref{boundsProp}(3) unless $p=2$ and $\frac{f}{g}=1$, for which we need to show that $\MCS(F_4(2))\geq 4$. Let $s\in F_4(2)$. Unless $s$ is of order $2$ or $3$, it is clear that $|\C_{F_4(2)}(s)|\geq 4$, and for elements of one of these two orders, it follows from the fact that $\min\{\nu_2(|F_4(2)|),\nu_3(|F_4(2)|)\}>1$. Now assume that $g\nmid f$. Then $p=2$, and $g=2g'$, where $g'\mid f$ and $\frac{f}{g'}$ is odd. By the final paragraph in Proposition \ref{hartleyProp}, we are done if we can show that either $\MCS(\leftidx{^2}F_4(2^{f/g'}))\geq2\frac{f}{g'}$ or the by Proposition \ref{boundsProp}(3) stronger inequality $2^{f/(2g')}\geq\frac{f}{g'}$ holds. The former is clearly true for $\frac{f}{g'}=1$, and the latter is true for $\frac{f}{g'}\geq 5$, whereas for $\frac{f}{g'}=3$, Proposition \ref{boundsProp}(3) yields $\MCS(\leftidx{^2}F_4(8))\geq2\sqrt{8}=5.6568\ldots$, and thus $\MCS(\leftidx{^2}F_4(8))\geq6$, as required.

\item Case: $S=\leftidx{^2}A_r(p^{2f})=\PSU_{r+1}(p^f)$, $r\geq 4$, so $|\Out(S)|=(r+1,p^f+1)\cdot 2f$ and $g\mid 2f$. We claim that then $\h(S)\leq\frac{10}{11}$, which is clear if $|\C_{\Aut(S)}|\geq\frac{11}{5}(r+1)f$. First, assume that $2\nmid g$. Then by Proposition \ref{hartleyProp}(3), it is sufficient to have $\MCS(\PGU_{r+1}(p^{f/g}))\geq\frac{11}{5}(r+1)\frac{f}{g}$, which was already checked to be true in Case 12. Now assume $2\mid g$. Then by Proposition \ref{hartleyProp}(4), it is sufficient to have $\MCS(\PGL_{r+1}(p^{2f/g}))\geq\frac{11}{5}(r+1)\frac{f}{g}=\frac{11}{10}(r+1)\frac{2f}{g}$, which also holds by Case 12.

\item Case: $S=\leftidx{^2}D_4(p^{2f})=\POmega_8^-(p^f)$, so $|\Out(S)|=(4,p^{4f}+1)\cdot 2f$ and $g\mid 2f$. We claim that then $\h(S)\leq\frac{12}{13}$. For $p=2$, where $|\Out(S)|=2f$ only, this can be verified using Proposition \ref{hartleyProp}(3,4) and the fact, already verified in Case 13, that $\MCS(\POmega_8^{\pm}(2^{f/g}))\geq\frac{13}{6}\frac{f}{g}$. The argument for $p\geq3$ is analogous, using the already established bound $\MCS(\POmega_8^{\pm}(p^{f/g}))\geq 26\frac{f}{g}$.

\item Case: $S=\leftidx{^3}D_4(p^{3f})$, so $|\Out(S)|=3f$ and $g\mid 3f$. We claim that then $\h(S)\leq\frac{12}{13}$, which can be verified using Proposition \ref{hartleyProp}(3,4) and the bounds $\MCS(\POmega_8^+(p^{3f/g}))\geq\frac{13}{6}\frac{3f}{g}=\frac{13}{2}\frac{f}{g}$ and $\MCS(\leftidx{^3}D_4(p^{3f/g}))\geq\frac{13}{4}\frac{f}{g}$, which were already established in Case 13.

\item Case: $S=\leftidx{^2}F_4(2^{2f'+1})$, so $|\Out(S)|=2f'+1=2f$ and $g\mid 2f'+1$. We claim that then $\h(S)\leq\frac{2}{5}$. As in Case 6, this holds by the last paragraph of Proposition \ref{hartleyProp} if $\MCS(\leftidx{^2}F_4(2^{(2f'+1)/g}))\geq\frac{5}{2}\cdot\frac{2f'+1}{g}$. For $\frac{2f'+1}{g}=1$, this clearly holds (noting that $\nu_2(|\leftidx{^2}F_4(2)|)=12>1$), and for $\frac{2f'+1}{g}\geq 3$, it holds by Proposition \ref{boundsProp}(3).

\item Case: $S=D_r(p^f)=\POmega_{2r}^+(p^f)$, $r\geq 5$, so $|\Out(S)|=(2,p-1)^2\cdot 2f$ and $g\mid f$. We claim that then $\h(S)\leq\frac{3}{4}$. First, assume $p=2$, so $|\Out(S)|=2f$ only. By Proposition \ref{hartleyProp}(1,2), it suffices to have $\MCS(\POmega_{2r}^{\pm}(2^{f/g}))\geq\frac{8}{3}\frac{f}{g}$, which holds by Proposition \ref{boundsProp}(3). Now assume $p\geq 3$. Then $|\Out(S)|=8f$, and so by Proposition \ref{hartleyProp}(1,2), we are satisfied if $\MCS(\PGO_{2r}^{\pm}(p^{f/g}))\geq\frac{32}{3}\frac{f}{g}$. Now by \cite[Theorem 6.13(2)]{FG12a}, we have

\begin{align*}
\MCS(\PGO_{2r}^{\pm}(p^{f/g})) &\geq p^{(r-1)f/g}\cdot\sqrt{\frac{1-1/p^{f/g}}{2\e(4+\log_{p^{f/g}}(4r))}} \\
&\geq p^{(r-1)f/g}\cdot\sqrt{\frac{1}{3\e(4+\log_3(4r))}},
\end{align*}

and this is greater than or equal to $\frac{32}{3}\frac{f}{g}$ if and only if

\begin{equation}\label{eq5}
p^{(r-1)f/g}\geq\frac{32}{3}\sqrt{3\e(4+\log_3(4r))}\frac{f}{g}.
\end{equation}

By an inductive argument, for each prime $p\geq 3$, Formula (\ref{eq5}) holds for all values of $\frac{f}{g}\geq 1$ if and only if it holds for $\frac{f}{g}=1$. Moreover, Formula (\ref{eq5}) with $\frac{f}{g}=1$ holds for all primes $p\geq 3$ if and only if it holds for $p=3$, so we are left with checking that for $r\geq 5$,

\begin{equation}\label{eq6}
3^{r-1}\geq\frac{32}{3}\sqrt{3\e(4+\log_3(4r))}.
\end{equation}

We verify this by induction on $r$: One readily checks that Formula (\ref{eq6}) holds for $r=5$, and upon replacing $r$ by $r+1$, the left-hand side of Formula (\ref{eq6}) grows by a factor of $3$, whereas the right-hand side only grows by a factor of $\sqrt{\frac{4+\log_3(4r+4)}{4+\log_3(4r)}}\leq 1+\frac{1}{4+\log_3(4r)}\leq 1+\frac{1}{4+\log_3(20)}=1.1486\ldots<3$.

\item Case: $S=\leftidx{^2}D_r(p^{2f})=\POmega_{2r}^-(p^f)$, $r\geq5$, so $|\Out(S)|=(4,p^{rf}+1)\cdot 2f$ and $g\mid 2f$. We claim that then $\h(S)\leq\frac{3}{4}$, and this follows from Proposition \ref{hartleyProp}(3,4) and the bounds $\MCS(\POmega_{2r}^{\pm}(2^{f/g}))\geq\frac{8}{3}\frac{f}{g}$ and $\MCS(\PGO_{2r}^{\pm}(p^{f/g}))\geq\frac{32}{3}\frac{f}{g}$ for $p\geq 3$ established in Case 19.

\item Case: $S=E_6(p^f)$, so $|\Out(S)|=(3,p^f-1)\cdot 2f$ and $g\mid f$. We claim that then $\h(S)\leq\frac{18}{19}$. We make a subcase distinction:

\begin{enumerate}
\item Subcase: $\alpha$ is not out-central. Then we are satisfied if $|\C_{\Aut(S)}(\alpha)|\geq\frac{19}{6}f$. By Proposition \ref{hartleyProp}(1,2), this holds if

\[
\min\{\MCS(\Inndiag(E_6(p^{f/g}))),\MCS(\Inndiag(\leftidx{^2}E_6(p^{2f/g})))\}\geq\frac{19}{6}\frac{f}{g},
\]

which holds by Proposition \ref{boundsProp}(3).
\item Subcase: $\alpha$ is out-central. Note that if $\gcd(3,p^f-1)=1$, then $|\Out(S)|=2f$ only, so that it is sufficient to have $|\C_{\Aut(S)}(\alpha)|\geq\frac{19}{9}f$, which can be proved as in the previous subcase. We may thus henceforth assume $\gcd(3,p^f-1)=3$, so that $|\Out(S)|=6f$. Since $\alpha$ is out-central, by the structure of $\Out(S)$ \cite[Theorem 2.5.12(b,c,g), p.~58]{GLS98a}, the image of $\alpha$ in $\Out(S)$ has trivial outer diagonal component, and so the $S$-coset of $\alpha$ in $\Aut(S)$ contains an element $\beta\in\Phi_S\Gamma_S$. By \cite[Theorem 2.5.12(e)]{GLS98a}, $\beta$ commutes with all the $2f$ elements of $\Phi_S\Gamma_S$. These elements form a transversal for $\Inndiag(S)$ in $\Aut(S)$, and so by Proposition \ref{hartleyProp}(1,2), we have

\begin{align*}
&|\C_{\Aut(S)}(\beta)|\geq 2f\cdot|\C_{\Inndiag(S)}(\beta^g)| \\
&\geq 2f\cdot\min\{\MCS(\Inndiag(E_6(p^{f/g}))),\MCS(\Inndiag(\leftidx{^2}E_6(p^{2f/g})))\} \\
&\geq 2f\cdot\frac{19}{6}\frac{f}{g}>6f.
\end{align*}

Hence $|\C_{\Aut(S)}(\alpha)|>6f$ as well. Since $|\C_{\Aut(S)}(\alpha)|$ is also an integer multiple of $g\cdot|\C_{\Inndiag(S)}(\alpha)|$, we can conclude with the following subsubcase distinction:

\begin{itemize}
\item Subsubcase: $|\C_{\Aut(S)}(\alpha)|\geq 2g\cdot|\C_{\Inndiag(S)}(\alpha)|$. Then by the computations in the previous subcase, $|\C_{\Aut(S)}(\alpha)|\geq 2g\cdot\frac{19}{6}\frac{f}{g}=\frac{19}{3}f=\frac{19}{18}|\Out(S)|$.
\item Subsubcase: $|\C_{\Aut(S)}(\alpha)|=g\cdot|\C_{\Inndiag(S)}(\alpha)|$. First, assume $\frac{f}{g}\geq 4$. Then by Propositions \ref{hartleyProp}(1,2) and \ref{boundsProp}(3), we have $|\C_{\Aut(S)}(\alpha)|\geq g\cdot 2p^{f/g}\geq g\cdot 2\cdot 4\frac{f}{g}=8f\geq\frac{19}{3}f=\frac{19}{18}|\Out(S)|$. Now assume $\frac{f}{g}\leq3$, say $g=\frac{f}{t}$ with $t\in\{1,2,3\}$. Then $|\C_{\Aut(S)}(\alpha)|=\frac{f}{t}\cdot m$, where $m$ is an integer such that $\frac{m}{t}>6$. It follows that $|\C_{\Aut(S)}(\alpha)|\geq\frac{19}{3}f=\frac{19}{18}|\Out(S)|$.
\end{itemize}
\end{enumerate}

\item Case: $S=\leftidx{^2}E_6(p^{2f})$, so $|\Out(S)|=(3,p^f+1)\cdot 2f$ and $g\mid 2f$. We claim that then $\h(S)\leq\frac{18}{19}$. We make a subcase distinction:

\begin{enumerate}
\item Subcase: $\alpha$ is not out-central. Then we are satisfied if $|\C_{\Aut(S)}(\alpha)|\geq\frac{19}{6}f$. If $2\nmid g$, then by Proposition \ref{hartleyProp}(3), we are done if

\[
\MCS(\Inndiag(\leftidx{^2}E_6(p^{2f/g}))\geq\frac{19}{6}\frac{f}{g},
\]

which was already verified in Subcase (a) of Case 21. Similarly, if $2\mid g$, then by Proposition \ref{hartleyProp}(4), we are done if

\[
\MCS(\Inndiag(E_6(p^{2f/g})))\geq\frac{19}{6}\frac{f}{g}=\frac{19}{12}\frac{2f}{g},
\]

and this also holds by Subcase (a) of Case 21.
\item Subcase: $\alpha$ is out-central. As in Subcase (b) of Case 21, we may assume that $\gcd(3,p^f+1)=3$, so that $|\Out(S)|=6f$. Since $\alpha$ is out-central, the outer diagonal component of the image of $\alpha$ in $\Out(S)$ must be trivial, so that the coset $S\alpha$ contains a unique element $\beta\in\Phi_S$. Now $\beta$ commutes with all the $2f$ elements of $\Phi_S$, which form a transversal for $\Inndiag(S)$ in $\Aut(S)$. Hence by Proposition \ref{hartleyProp}(3,4) and the computations in Case 21, we get that $|\C_{\Aut(S)}(\beta)|\geq 2f\cdot|\C_{\Inndiag(S)}(\beta)|\geq 2f\cdot 4>6f=|\Out(S)|$. Hence $|\C_{\Aut(S)}(\alpha)|>6f$ as well, and we can conclude with a subsubcase distinction as at the end of Subcase (b) of Case 21.
\end{enumerate}

\item Case: $S=E_7(p^f)$, so $|\Out(S)|=(2,p^f-1)\cdot f$ and $g\mid f$. We claim that then $\h(S)\leq\frac{13}{32}$. Indeed, using Proposition \ref{hartleyProp}(1) and \cite[Lemma 6.14]{FG12a}, we get that $|\C_{\Aut(S)}(\alpha)|\geq g\cdot\MCS(E_7(p^{f/g}))\geq g\cdot\frac{p^{7f/g}}{26}\geq g\cdot\frac{128f/g}{26}=\frac{64}{13}f\geq\frac{32}{13}|\Out(S)|$.

\item Case: $S=E_8(p^f)$, so $|\Out(S)|=f$ and $g\mid f$. We claim that then $\h(S)\leq\frac{13}{128}$. Indeed, using Proposition \ref{hartleyProp}(1) and \cite[Lemma 6.14]{FG12a}, we get that $|\C_{\Aut(S)}(\alpha)|\geq g\cdot\MCS(E_8(p^{f/g}))\geq g\cdot\frac{p^{8f/g}}{26}\geq g\cdot\frac{256f/g}{26}=\frac{128}{13}f=\frac{128}{13}|\Out(S)|$.
\end{enumerate}
\end{proof}

\section{Proof of Theorem \ref{mainTheo}}\label{sec3}

\subsection{Proof of Theorem \ref{mainTheo}(1)}\label{subsec3P1}

We show the contraposition: If $G$ is a finite nonsolvable group, then $\maol(G)\leq\frac{18}{19}$. Indeed, let $S$ be a nonabelian composition factor of $G$. By Lemma \ref{mainLem1}(2), $G$ has a semisimple characteristic quotient $H$ such that $\Soc(H)\cong S^n$ for some $n\in\IN^+$. By Lemmas \ref{mainLem2}(2) and \ref{mainLem3}(2), $\maol(H)\leq\h(S)$, and so, by Lemmas \ref{mainLem4}(1) and \ref{mainLem1}(1), we conclude that $\maol(G)\leq\maol(H)\leq\h(S)\leq\frac{18}{19}$.\qed

\subsection{Proof of Theorem \ref{mainTheo}(2)}\label{subsec3P2}

Fix $\rho\in\left(0,1\right]$, let $G$ be a finite group with $\maol(G)\geq\rho$, and let $S$ be a nonabelian composition factor of $G$. We need to show that $|S|$ is bounded in terms of $\rho$. Just as in the proof of Theorem \ref{mainTheo}(1), we find that $G$ has a semisimple characteristic quotient $H$ such that $\Soc(H)\cong S^n$ for some $n\in\IN^+$, so that $\rho\leq\maol(G)\leq\maol(H)\leq\h(S)$. By Lemma \ref{mainLem4}(2), this excludes alternating groups of large degree as well as simple Lie type groups $\leftidx{^t}X_r(p^{ft})$ where either the untwisted Lie rank $r$ or the defining characteristic $p$ is large as possibilities for $S$. We are therefore done if we can show the following:

\begin{propposition}\label{fProp}
Let $\leftidx{^t}X_r$ resp.~$p$ be a fixed Lie symbol resp.~prime. Then as $f\to\infty$, $\sup_{H\text{ fin.~semisimple gp.~with }\Soc(H)\text{ a power of }\leftidx{^t}X_r(p^{ft})}{\maol(H)}\to0$.
\end{propposition}

For proving Proposition \ref{fProp}, we introduce one more concept and formulate a simple lemma concerning this concept.

\begin{deffinition}\label{coarseDef}
Let $S$ be a nonabelian finite simple group, $c$ a conjugacy class in $\Aut(S)$, $\pi:\Aut(S)\rightarrow\Aut(S)/\Inndiag(S)\cong\Phi_S\Gamma_S$ the canonical projection.

\begin{enumerate}
\item The term \emph{coarse $S$-type} is a synonym for \enquote{element of $\Aut(S)/\Inndiag(S)$}.
\item We call the unique element of the $\Aut(S)/\Inndiag(S)$-conjugacy class $\pi[c]$ the \emph{coarse $S$-type of $c$}.
\end{enumerate}
\end{deffinition}

In Definition \ref{coarseDef}(2), we are using that by \cite[Theorem 2.5.12(e), p.~58]{GLS98a}, the group $\Aut(S)/\Inndiag(S)$ is always abelian. Note that if two conjugacy classes in $\Aut(S)$ have the same $S$-type in the sense of Definition \ref{typeDef}(2), then they also have the same coarse $S$-type, but not vice versa.

\begin{deffinition}\label{coarseSetDef}
Let $S$ be a nonabelian finite simple group, $n\in\IN^+$, let $\vec{\alpha}=(\alpha_1,\ldots,\alpha_n)\sigma\in\Aut(S)\wr\Sym_n$, and let $\pi:\Aut(S)\rightarrow\Aut(S)/\Inndiag(S)$ be the canonical projection. We define $\CT(\vec{\alpha})$ as the following set (\emph{not} multiset):

\[
\CT(\vec{\alpha}):=\{\tau\in\Aut(S)/\Inndiag(S)\mid\tau\in\pi[\bcpc_{\zeta}(\vec{\alpha})]\text{ for some cycle }\zeta\text{ of }\sigma\}.
\]
\end{deffinition}

Let us now turn to the aforementioned lemma; in its second statement, we write the group $\Aut(S)/\Inndiag(S)$ multiplicatively, although it is abelian.

\begin{lemmma}\label{coarseSetLem}
Let $S$ be a nonabelian finite simple group, $n\in\IN^+$, $H$ a finite semisimple group with $\Soc(H)\cong S^n$ and $\vec{\alpha},\vec{\beta}\in H$.

\begin{enumerate}
\item If $\vec{\alpha}$ and $\vec{\beta}$ lie in the same $\Aut(H)$-orbit, then $\CT(\vec{\alpha})=\CT(\vec{\beta})$.
\item Let $\sigma$ be the permutation part of $\vec{\alpha}$, and let $k\in\IZ$ with $\gcd(k,\ord(\sigma))=1$. Then $\CT(\vec{\alpha}^k)=\{t^k\mid t\in\CT(\vec{\alpha})\}$.
\end{enumerate}
\end{lemmma}

\begin{proof}
For (1): This is clear from Lemma \ref{mainLem2}(1).

For (2): Write $\vec{\alpha}=(\alpha_1,\ldots,\alpha_n)\sigma$. The permutation part $\sigma^k$ of $\vec{\alpha}^k$ has the same cycle support sets (in particular the same cycle type) as $\sigma$, and for each cycle $\zeta=(i_1,\ldots,i_l)$ of $\vec{\alpha}$, the entries of the tuple part of $\vec{\alpha}^k$ corresponding to indices from $\supp(\zeta)=\{i_1,\ldots,i_l\}=\supp(\zeta^k)$ are $k$-factor products of such entries of $(\alpha_1,\ldots,\alpha_n)$. Moreover, by symmetry, each entry $\alpha_{i_j}$, $j\in\{1,\ldots,l\}$, formally occurs $k$ times in total as a factor in one of the entries of the tuple part of $\vec{\alpha}^k$ corresponding to an index from $\supp(\zeta)=\supp(\zeta^k)$. The assertion now follows by the commutativity of the group $\Aut(S)/\Inndiag(S)$.
\end{proof}

We are now ready to prove Proposition \ref{fProp}.

\begin{proof}[Proof of Proposition \ref{fProp}]
Let $H$ be a finite semisimple group with $\Soc(H)$ a power of $\leftidx{^t}X_r(p^{ft})=:S$, say $\Soc(H)\cong S^n$. We show that for each $\epsilon>0$, if $f$ is large enough (where this notion of \enquote{large enough} may depend on $\epsilon$, $r$ and $p$, but not on $n$), then $\maol(H)<\epsilon$. Let $\vec{\alpha}=(\alpha_1,\ldots,\alpha_n)\sigma\in H$. We make a case distinction:

\begin{enumerate}
\item Case: $|\CT(\vec{\alpha})|>\log_{18/19}(\epsilon)$. Then, using that in the product upper bound on $\frac{1}{|H|}|\vec{\alpha}^{\Aut(H)}|$ given in Lemma \ref{mainLem2}(2), the number of factors is bounded from below by $|\CT(\vec{\alpha})|$ and that by Lemmas \ref{mainLem3}(2) and \ref{mainLem4}(1), each of these factors is at most $\frac{18}{19}$, we conclude that $\frac{1}{|H|}|\vec{\alpha}^{\Aut(H)}|\leq(\frac{18}{19})^{|\CT(\vec{\alpha})|}<\epsilon$.
\item Case: $|\CT(\vec{\alpha})|\leq\log_{18/19}(\epsilon)$. We make a subcase distinction:

\begin{enumerate}
\item Subcase: There is a cycle $\zeta$ of $\sigma$ such that the common field (resp.~graph-field, if $\leftidx{^t}X_r$ is one of $B_2$, $G_2$ or $F_4$) component order $g_{\zeta}$ of any of the elements of the $\Aut(S)$-conjugacy class $\bcpc_{\zeta}(\vec{\alpha})$ satisfies $\frac{f}{g_{\zeta}}>5\frac{\sqrt{2}}{\sqrt{2}-1}(r+1)\epsilon^{-1}$. Let $l$ be the length of $\zeta$, and let $\tau$ be the $S$-type of $\bcpc_{\zeta}(\vec{\alpha})$. Then by Lemma \ref{mainLem2}(2), $\frac{1}{|H|}|\vec{\alpha}^{\Aut(H)}|\leq r(M_l^{\tau}(\vec{\alpha}))$, which by Lemma \ref{mainLem3}(2) and using the notation from Definition \ref{rDef}(4) is bounded from above by the maximum of the $\rho(c_i^{(\tau)})$, $i=1,\ldots,k(\tau)$. Now fix $i\in\{1,\ldots,k(\tau)\}$, and let $\beta\in c_i^{(\tau)}$. Since $|\Out(S)|=|\Out(\leftidx{^t}X_r(p^{ft}))|\leq 5(r+1)f$, we are done if we can show that $|\C_{\Aut(S)}(\beta)|>5(r+1)f\epsilon^{-1}$. But by Propositions \ref{hartleyProp} and \ref{boundsProp}, we have

\begin{align*}
|\C_{\Aut(S))}(\beta)| &\geq g_{\zeta}\cdot(p^{f/g_{\zeta}}-1)\geq g_{\zeta}\cdot\frac{\sqrt{2}-1}{\sqrt{2}}p^{f/g_{\zeta}}\geq \frac{\sqrt{2}-1}{\sqrt{2}}g_{\zeta}(\frac{f}{g_{\zeta}})^2 \\
&=\frac{\sqrt{2}-1}{\sqrt{2}}f\cdot\frac{f}{g_{\zeta}}>5(r+1)f\epsilon^{-1},
\end{align*}

as required.

\item Subcase: For every cycle $\zeta$ of $\sigma$, the common field (resp.~graph-field) component order $g_{\zeta}$ of any of the elements of $\bcpc_{\zeta}(\vec{\alpha})$ satisfies $g_{\zeta}\geq\frac{f}{5\frac{\sqrt{2}}{\sqrt{2}-1}(r+1)\epsilon^{-1}}$. Observe that $g_{\zeta}\mid\ord(\tau_{\zeta})\mid o_{\zeta}\mid\ord(\vec{\alpha})$, where $\tau_{\zeta}$ is the coarse $S$-type of $\bcpc_{\zeta}(\vec{\alpha})$, and $o_{\zeta}$ is the common order of any of the elements of $\bcpc_{\zeta}(\vec{\alpha})$; for the first divisibility relation, use \cite[Theorem 2.5.12(e), p.~58]{GLS98a}, and for the last, note that if $l$ is the length of $\zeta$, then the tuple part of $\vec{\alpha}^l$ has elements from $\bcpc_{\zeta}(\vec{\alpha})$ at the entries corresponding to indices from $\supp(\zeta)$, while its permutation part fixes all such indices, whence $o_{\zeta}\mid\ord(\vec{\alpha}^l)\mid\ord(\vec{\alpha})$. In particular, by the case and subcase assumptions, $\CT(\vec{\alpha})$ is a subset of the field-graph automorphism group $\Phi_S\Gamma_S$ whose size is bounded from above by $\log_{18/19}(\epsilon)=:C_1(\epsilon)$ and whose elements all have order at least $\frac{f}{5\frac{\sqrt{2}}{\sqrt{2}-1}(r+1)\epsilon^{-1}}=:\frac{f}{C_2(\epsilon,r)}$. Define $O$ to be the maximum value of $\ord(\tau_{\zeta})$, where $\zeta$ ranges over the cycles of $\sigma$, and fix a cycle $\zeta$ of $\sigma$ such that $\ord(\tau_{\zeta})=O$. Set $M:=\CT(\vec{\alpha})\cap\langle\tau_{\zeta}\rangle$. Consider the natural action of the group of units $(\IZ/O\IZ)^{\ast}$ on the subsets of $\langle\tau_{\zeta}\rangle\cong\IZ/O\IZ$. Since $M$ contains the generator $\tau_{\zeta}$, the stabilizer of $M$ under this action has order at most $|M|\leq C_1(\epsilon)$. Fix a transversal $u_1,\ldots,u_m$ of $\Stab_{(\IZ/O\IZ)^{\ast}}(M)$ in $(\IZ/O\IZ)^{\ast}$, where $m\geq\frac{\phi(f/C_2(\epsilon,r))}{C_1(\epsilon)}$. For $i=1,\ldots,m$, let $k_i$ be a lift of $u_i$ to $(\IZ/\lcm(\ord(\vec{\alpha},\exp(\Phi_S\Gamma_S))))^{\ast}$. Then by Lemma \ref{coarseSetLem}(2), for each $i=1,\ldots,m$, $\CT(\vec{\alpha}^{k_i})\cap\langle\tau_{\zeta}\rangle=\{\tau^{k_i}\mid \tau\in M\}$, and so the sets $\CT(\vec{\alpha}^{k_i})$, $i=1,\ldots,m$, are pairwise distinct. By Lemma \ref{coarseSetLem}(1), the powers $\vec{\alpha}^{k_i}$, $i=1,\ldots,m$, therefore lie in pairwise distinct automorphism orbits in $H$, which are all of the same length as $\vec{\alpha}^{\Aut(H)}$. It follows that if $f$ is so large (in dependence of $\epsilon$ and $r$) that $\frac{\phi(f/C_2(\epsilon,r))}{C_1(\epsilon)}>\epsilon^{-1}$, then $\frac{1}{|H|}|\vec{\alpha}^{\Aut(H)}|\leq\frac{1}{m}<\epsilon$, as required.
\end{enumerate}
\end{enumerate}
\end{proof}

\subsection{Proof of Theorem \ref{mainTheo}(3)}\label{subsec3P3}

We claim that any number in $\left(0,1\right]$ that is strictly smaller than $\maol(\Aut(S))$ does the job as a choice for $c(S)$; so one could, for example, choose $c(S):=\frac{\maol(\Aut(S))}{2}$.

To see that this holds, consider, for each prime $p$, the finite semisimple group $H_p:=\Aut(S)^p\rtimes\langle\sigma\rangle=\Aut(S)\wr\langle\sigma\rangle$, where $\sigma\in\Sym_p$ is a $p$-cycle. We claim that $\maol(H_p)\geq(1-\frac{1}{p})\maol(\Aut(S))$, which converges to $\maol(\Aut(S))$ from below as $p\to\infty$ (so this is sufficient).

In order to verify this bound, observe that $\Aut(H_p)=\Aut(S)\wr\N_{\Sym_p}(\langle\sigma\rangle)$. Consider the element $\vec{\alpha}:=(\alpha_1,\id_S,\ldots,\id_S)\sigma$, where $\alpha_1\in\Aut(S)$ is chosen such that its automorphism orbit in $\Aut(S)$ is of the maximum possible proportion $\maol(\Aut(S))$. Since $\N_{\Sym_p}(\langle\sigma\rangle)$ acts transitively on the set of nontrivial elements of $\langle\sigma\rangle$, we find that $\vec{\alpha}^{\Aut(H_p)}$ contains elements with all nontrivial elements of $\langle\sigma\rangle$ as possible permutation parts. Therefore, $|\vec{\alpha}^{\Aut(H_p)}|=(p-1)\cdot|\vec{\alpha}^{\Aut(H_p)}\cap\Aut(S)^p\sigma|$, and since $\langle\sigma\rangle$ is a self-centralizing subgroup of $\Sym_p$, the stabilizer of the set $\vec{\alpha}^{\Aut(H_p)}\cap\Aut(S)^p\sigma$ in $\Aut(S^p)=\Aut(S)\wr\Sym_p$ is just $H_p$, which is contained in $\Aut(H_p)$. Hence $\vec{\alpha}^{\Aut(H_p)}\cap\Aut(S)^p\sigma=\vec{\alpha}^{\Aut(S^p)}\cap\Aut(S)^p$, which by Lemma \ref{mainLem2}(1) consists of just those elements $(\beta_1,\ldots,\beta_p)\sigma\in\Aut(S^p)$ such that $\beta_p\cdots\beta_1\in\bcpc_{\sigma}(\vec{\alpha})=\alpha_1^{\Aut(S)}$. Hence $|\vec{\alpha}^{\Aut(H_p)}|=(p-1)\cdot|\alpha_1^{\Aut(S)}||\Aut(S)|^{p-1}$, whence $\vec{\alpha}^{\Aut(H_p)}$ is of proportion $(1-\frac{1}{p})\maol(\Aut(S))$ in $H_p$, as required.\qed

\section{Concluding remarks}\label{sec4}

We conclude this paper with some related open questions and problems for further research. Probably the most natural thing to ask in light of our main results is whether the constant $\frac{18}{19}$ in Theorem \ref{mainTheo}(1) can be reduced and, more strongly, what the optimal value for it is. In this context, we note without proof the following result, which was checked by the author with a combination of GAP \cite{GAP4} computations, looking up information in the ATLAS of Finite Group Representations \cite{ATLAS} and using a few ad hoc arguments:

\begin{proposition}\label{smallExProp}
The following hold:

\begin{enumerate}
\item Let $H$ be a nontrivial finite semisimple group such that $\Soc(H)$ is characteristically simple and $|\Soc(H)|\leq10^5$. Then $\maol(H)\leq\frac{3}{7}$.
\item Let $G$ be a finite nonsolvable group with $|G|\leq10^5$. Then $\maol(G)\leq\frac{3}{7}=\maol(\PSL_2(8))$.
\end{enumerate}
\end{proposition}

This motivates the following question, a positive answer to which would settle the search for the optimal constant in Theorem \ref{mainTheo}(1):

\begin{question}\label{ques1}
Is it true that $\maol(G)\leq\frac{3}{7}$ for all finite nonsolvable groups $G$?
\end{question}

Concerning the second open question, we note that while the examples which we gave in Remark \ref{mainRem}(3) show that commutativity of $G$ cannot be forced under an assumption on $G$ of the form $\maol(G)\geq\rho$ with $\rho\in\left(0,1\right)$, they are all nilpotent, and the author has not yet found a sequence of nonnilpotent finite groups whose $\maol$-values converge to $1$. So the following may be worth looking into:

\begin{question}\label{ques2}
Does there exist a constant $c\in\left(0,1\right)$ such that $\maol(G)\leq c$ for all finite nonnilpotent groups $G$?
\end{question}

We would like to propose one more question, which is motivated by the connection, mentioned in Remark \ref{mainRem}(2), between the results of this paper and \cite[Theorem 1.1.1]{Bor17a}. Consider the following constants:

\begin{notation}\label{cNot}
For $k\in\IN^+$, we set

\[
c_k:=\sup_{G\text{ fin.~nonsolvable gp.}}{\frac{1}{|G|}\max_{\alpha_1,\ldots,\alpha_k\in\Aut(G)}\max_{g\in G}{|g^{\langle\alpha_1,\ldots,\alpha_k\rangle}|}}.
\]
\end{notation}

So $c_k$ is a constant upper bound on the orbit proportions on finite nonsolvable groups that one can achieve by $k$-generated subgroups of the automorphism group. This has some relevance for the study of finite groups for use in pseudorandom number generation, as outlined in \cite[Subsection 4.1]{Bor16a}. In the classical setting, a pseudorandom number generator consists of the iteration of one fixed function on a state space, but one can also consider \enquote{random walk generators}, where in each step, a transformation randomly chosen from a fixed set of candidates is applied. The basic requirement that the orbits on the state space obtained this way should be \enquote{large} remains unchanged though, and our results show that if one wants orbits of proportion larger than $\frac{18}{19}$, achieved by randomly applying one of $k$ given automorphisms of a finite group $G$ starting from some element $g\in G$ as the seed, one must use a solvable $G$.

It is natural to ask if the constant $\frac{18}{19}$ can be improved on in this context as well. Clearly, the sequence $(c_k)_{k\in\IN^+}$ is non-decreasing, and since, by \cite[Theorem 1]{DL95a}, $\Aut(\PSL_2(8))$ is $2$-generated, we have $c_2\geq\frac{3}{7}$, so if the answer to Question \ref{ques1} is \enquote{yes}, then $c_k=\frac{3}{7}$ for all $k\geq2$, but if the answer is \enquote{no}, the following problem may also be worth studying:

\begin{problem}\label{problem1}
Determine the precise value of (or at least give better bounds on) $c_k$ for some given $k\geq 2$. Also, determine if the sequence $(c_k)_{k\in\IN^+}$ is eventually constant.
\end{problem}

\section{Acknowledgements}\label{sec5}

The author would like to thank Michael Giudici and Cheryl Praeger for some helpful discussions while he was working on this paper.

\end{document}